\documentclass[11pt]{article}
\usepackage{amsmath}
\usepackage{amsfonts}
\usepackage{amssymb}
\usepackage{color}
\newtheorem{theorem}{Theorem}
\newtheorem{corollary}{Corollary}
\newtheorem{lemma}{Lemma}
\newtheorem{proposition}{Proposition}

\newtheorem{definition}{Definition}
\newtheorem{example}{Example}

\newtheorem{remark}{Remark}

\usepackage[top=3.5cm, bottom=3.5cm, left=2cm,right=2cm]{geometry}
\usepackage{amsbsy,amsfonts,amsmath,amssymb,enumerate,epsfig,graphicx,rotating,subfigure}

\font\QEDlogofont=msam10 at 10pt

\def\QEDblogo{\hbox{\QEDlogofont\char'004}}
\newif\ifnologo\nologofalse
\newif\iflogo
\newif\ifblogo\blogofalse
\newif\iftopprhead\topprheadfalse

\def\prooffont{\normalsize}
\newenvironment{proof}{\par\addvspace{6pt plus2pt}
\par
\noindent\prooffont{\bf\em Proof:}\hskip6pt\ignorespaces}{%
   \ifblogo\hskip1.2pt
            \blacksquare
   \else
   \ifnologo
   \else
   \hfill
            \QEDblogo
   \fi\fi
\par\addvspace{6pt plus2pt}\global\topprheadfalse}%

\usepackage{amsbsy,amsfonts,amsmath,amssymb,enumerate,epsfig,graphicx}

\begin{document}

\begin{center}

\begin{large}
{\bf ON A CLASS OF POLYNOMIALS GENERATED BY $F(xt-R(t))$}
\end{large}
\vspace{10pt}

{\bf Mohammed Mesk $^{\rm a,b,}$\footnote{Email:
m\_mesk@yahoo.fr} and Mohammed Brahim Zahaf $^{\rm
a,c,}$\footnote{Email:
m\_b\_zahaf@yahoo.fr}}\vspace{6pt}
\\ \vspace{6pt} $^{\rm a}${\em Laboratoire d'Analyse Non Lin\'eaire et Math\'ematiques Appliqu\'ees,\break
Universit\'e de Tlemcen, BP 119,  13000-Tlemcen, Alg\'erie.
}\\

\vspace{6pt} $^{\rm b}${\em D\'epartement d'\'ecologie et environnement,
\break
Universit\'e de Tlemcen, P\^ole 2, BP 119,  13000-Tlemcen, Alg\'erie.
}\\

\vspace{6pt} $^{\rm c}${\em D\'epartement de Math\'ematiques,
Facult\'e des sciences, \break
Universit\'e de Tlemcen, BP 119,  13000-Tlemcen, Alg\'erie.
}\\

\end{center}
\vspace{1cm}
\begin{abstract}
We investigate polynomial sets $\{P_n\}_{n\geq 0}$ with generating power series of the form $F(xt-R(t))$ and satisfying, for $n\geq 0$, the $(d+1)$-order recursion $xP_n(x)=P_{n+1}(x)+\sum_{l=0}^{d}\gamma_{n}^{l}P_{n-l}(x)$, where $\{\gamma_{n}^{l}\}$ is  a complex sequence for  $0\leq l\leq d$, $P_0(x)=1$ and $P_n(x)=0$ for all negative integer $n$. We show that the formal power series $R(t)$ is a polynomial of degree at most  $d+1$ if certain coefficients of $R(t)$ are null or if $F(t)$ is a generalized hypergeometric series. Moreover, for the $d$-symmetric case we demonstrate that $R(t)$ is the monomial of degree $d+1$ and $F(t)$ is  expressed by hypergeometric series.
\vspace{6pt}


\vspace{6pt}

\end{abstract}

\section{Introduction}

In \cite{alsalam,Bachhaus,Bencheikh} the authors used different methods to show that the orthogonal polynomials defined by a generating function of the form $F(xt-\alpha t^2)$ are the ultraspherical and Hermite polynomials. On the other hand, the author in \cite{anshelev}  found (even if $F$ is a formal power series) that the orthogonal polynomials are the ultraspherical, Hermite and Chebychev polynomials of the first kind. Motivated by the problem, posed in \cite{anshelev}, of describing (all or just orthogonal) polynomials with generating functions $F(xU(t)-R(t))$ we have generalized in \cite{meskzahaf} the  above results  by proving the following: 
\begin{theorem} \cite{meskzahaf}\label{thmm0}
Let $F(t)=\sum_{n\geq 0}\alpha_nt^n$ and $R(t)=\sum_{n\geq
1}R_nt^n/n$ be formal power series where $\{\alpha_n\}$ and
$\{R_n\}$ are complex sequences with $\alpha_0=1$ and $R_1=0$.
Define the polynomial set $\{P_n\}_{n\geq 0}$ by
\begin{equation}\label{gf00}
F(xt-R(t))=\sum_{n\geq 0}\alpha_nP_n(x)t^n.
\end{equation}
If this polynomial set (which is automatically monic) satisfies the three-term recursion
relation
\begin{eqnarray}\label{gf6}
\left\{
\begin{array}{l}
xP_n(x)=P_{n+1}(x)+\beta_nP_n(x)+\omega_nP_{n-1}(x),\quad n\geq
0,\\
P_{-1}(x)=0,\;\;P_0(x)=1
\end{array}
\right.
\end{eqnarray} 
where $\{\beta_n\}$ and $\{\omega_n\}$ are complex sequences,
then we have:\\
a) If $R_2=0$ and $\alpha_n\neq 0$ for $n\geq 1$, then $R(t)=0$, $F(t)$ is arbitrary and $F(xt)=\sum_{n\geq 0}\alpha_nx^nt^n$
generates the monomials $\{x^n\}_{n\geq 0}$.\\
b) If $\alpha_1R_2\neq 0,$ then $R(t)=R_2t^2/2$ and the polynomial sets $\{P_n\}_{n\geq 0}$ are the rescaled ultraspherical, Hermite and Chebychev polynomials of the first kind.
\end{theorem}
 
Note that, the polynomials in Theorem 1 which satisfy a three term recursion with complex coefficients are not necessary orthogonal with respect to a moment functional $\mathcal{L}$, i.e. for all non negative integers $m,n$; $\mathcal{L}(P_m(x)P_n(x))=0$ if $m\neq n$ and $\mathcal{L}(P^2_n(x))\neq 0$, see Definition 2.2 in \cite{Chihara}.

\begin{remark}$\text{ }$\\
a) A polynomial set PS, $\{P_n\}_{n\geq 0}$, is such that $degree(P_n)=n$,\; $n\geq 0$.\\
b) A PS is called a monic PS if $P_n(x)=x^n+\cdots$, for $\; n\geq 0$.\\
c) The choice $\alpha_0=1$ and $R_1=0$ comes from the fact that the generating function $\gamma_1+\gamma_2F((x+R_1)t-R(t))=\gamma_1+\gamma_2\sum_{n\geq 0}\alpha_nP_n(x+R_1)t^n$, with $\gamma_1$ and $\gamma_2$ constants, is also of type \eqref{gf00}.
\end{remark}

In the present paper, we are interested in monic PSs generated by \eqref{gf00} (with $F(t)$ and $R(t)$ as in Theorem 1) and satisfying higher order recurrence relations \eqref{dorth0}. For this purpose, we adopt the following definitions:
\begin{definition}\label{def1}
Let $d\in \mathbb{N}$. A PS $\{Q_n\}_{n\geq 0}$ is called a $d$-polynomial set $d$-PS if its corresponding monic PS $\{P_n\}_{n\geq 0}$, defined by  $P_n(x)=(\lim_{x\rightarrow+\infty}x^{-n}Q_n(x))^{-1}Q_n(x)$, $n\geq 0$, satisfies the $(d+1)$-order recurrence relation:

\begin{eqnarray}\label{dorth0}
\left\{
\begin{array}{l}
xP_n(x)=P_{n+1}(x)+\sum_{l=0}^{d}\gamma_{n}^{l}P_{n-l}(x),\quad n\geq
0,\\
P_0(x)=1,\;\;P_{-l}(x)=0,\;\;1\leq l\leq d
\end{array}
\right.
\end{eqnarray}
where
\begin{equation}\label{dorth01}
\{\gamma_{n}^{l}\}_{n\geq 0}, 0\leq l\leq d, \text{are complex sequences}
\end{equation}
and
\begin{equation}\label{dorth02}
\{\gamma_{n}^{d}\}_{n\geq d} \;\text{is not the null sequence, for}\;\; d\geq 1.
\end{equation}
\end{definition}
\begin{definition}\label{def2}
Let $\{P_n\}_{n\geq 0}$ be a $d$-PS. If the PS of the derivatives $\{(n+1)^{-1}P_{n+1}^{'}\}_{n\geq 0}$ is also a $d$-PS, then $\{P_n\}_{n\geq 0}$ is called a classical $d$-PS.
\end{definition}
\begin{definition}\label{def3}\cite{MaroniDouak}
Let $\omega=\exp(2i\pi/(d+1))$, where $i^2=-1$.
The PS $\{P_n\}_{n\geq 0}$ is called $d$-symmetric if it fulfils:
\begin{equation}\label{dorth03}
P_n(\omega x)=\omega^nP_n(x),\;n\geq 0.
\end{equation}
\end{definition}
\begin{remark}\text{In Definition~\ref{def1}:}\\
a) For $d\geq 1$, the first terms $\{\gamma_{n}^{l}\}_{0\leq n<l\leq d}$ of the sequences $\{\gamma_{n}^{l}\}_{n\geq 0}$, $1\leq l\leq d$, are arbitrary.\\
b) For $d=0$, \eqref{dorth0} becomes $xP_n(x)=P_{n+1}(x)+\gamma_{n}^{0}P_{n}(x), n\geq 0,$ with $P_0(x)=1$. Here $\{\gamma_{n}^{0}\}_{n\geq 0}$ can be the null sequence, so the set of monomials is a $0$-PS.
\end{remark}
An interesting class of $d$-PSs characterized by \eqref{dorth0}, with the additional condition $\gamma_{n}^{d}\neq 0$ for $n\geq d$, are the $d$-orthogonal polynomial sets $d$-OPSs \cite{Maroni,Iseghem}. In this context, the authors in \cite{Bencheikh} generalized the result stated in \cite{alsalam,Bachhaus} by showing the following:
\begin{theorem}\cite{Bencheikh}\label{thmm1}
The only $d$-OPSs generated by $G((d+1)xt-t^{d+1})$ are the classical d-symmetric polynomials.
\end{theorem}  
Another contribution concerns $d$-OPSs with generating functions of Sheffer type, i.e. of the form $A(t)\exp(xH(t))$. We have
\begin{theorem}\cite{Varma}\label{thmm2}
Let $\rho_d(t)=\sum_{k=0}^{d}\tilde{\rho}_kt^k$ be a polynomial of degree d $(\tilde{\rho}_d\neq 0)$ and $\sigma_{d+1}(t)=\sum_{k=0}^{d+1}\tilde{\sigma}_kt^k$ be a polynomial of degree less than or equal to $d + 1$.
The only PSs, which are $d$-orthogonal and also Sheffer PS, are generated by
\begin{equation}\label{dorth002}
\exp\left(\int_{0}^{t}\frac{\rho_d(s)}{\sigma_{d+1}(s)}ds\right)\exp\left(x\int_{0}^{t}\frac{1}{\sigma_{d+1}(s)}ds\right)=\sum_{k=0}^{\infty}P_n(x)\frac{t^n}{n!}
\end{equation}
with the conditions
\begin{equation}\label{dorth0002}
\tilde{\sigma}_0(n\tilde{\sigma}_0-\tilde{\rho}_d)\neq 0,\;\; n\geq 1.
\end{equation}
\end{theorem}  
Note that Theorem 3 characterizes also the $d$-OPSs with generating functions of the form $F(xH(t)-R(t))$ with $F(t)=\exp(t)$, since $A(t)\exp(xH(t))=F(xH(t)-R(t))$ where $R(t)=-\ln(A(t))$. For $H(t)=t$ we meet the Appell case with the Hermite  $d$-OPSs \cite{Douak96} generated by $\exp(xt-\bar{\rho}_{d+1}(t))$ where $\bar{\rho}_{d+1}(t)$ is a polynomial of degree $d+1$.

As a consequence of the results obtained in this paper, we give some generalizations of Theorems 1, 2 and the Appell case in Theorem 3 (see also \cite{Douak96}) to $d$-PS generated by \eqref{gf00}.

After this short  introduction, we give in section 2 some results for $d$-PSs generated by \eqref{gf00}. Then in section 3 we show that the only $d$-symmetric $d$-PSs generated by \eqref{gf00} are the classical d-symmetric polynomials. For this later case, we give in section 4 the $(d+1)$-order recurrence relation \eqref{dorth0} and the expression of $F(t)$ by means of hypergeometric functions.

\section{Some general results}

The results in this section concern all $d$-PSs generated by \eqref{gf00}. The central result is Proposition 2 below from which the other results arise. First we have

\begin{proposition}\label{prop1}\cite{meskzahaf}

Let  $\{P_n\}_{n\geq 0}$ be a PS generated by \eqref{gf00}.
Then we have

\begin{equation}\label{gf1} 
\alpha_n xP'_n(x)-\sum_{k=1}^{n-1}R_{k+1}\alpha_{n-k}P'_{n-k}(x)=n\alpha_nP_n(x), \;\; n\geq 1.
\end{equation}

\end{proposition}
Secondly 
\begin{proposition}\label{prop2}
Let  $\{P_n\}$ be a $d$-PS generated by \eqref{gf00} and satisfying \eqref{dorth0}, \eqref{dorth01} and \eqref{dorth02}, with $\alpha_n\neq 0$ for $n\geq 1$. Putting  $$a_n=\frac{\alpha_n}{\alpha_{n+1}}, \;(n\geq
0)\text{ and } c_{n}^{l}=\frac{\alpha_n}{\alpha_{n-l}}\,\gamma_{n}^{l},\;\; (1\leq l\leq d,\;\;n\geq l),$$ 

then we have:

a) 

\begin{equation}\label{gf9}
\gamma_{n}^{0}=0,\;\;\text{  for } n\geq 0.
\end{equation}

b)
\begin{equation}\label{gf10}
\gamma_{n}^{1} =\frac{R_2}{2}(na_n-(n-1)a_{n-1}),\;\;\text{ for } n\geq 1.
\end{equation}
or equivalently
\begin{equation}\label{cgf10}
c_{n}^{1} =\frac{R_2}{2}\left(n\frac{a_n}{a_{n-1}}-(n-1)\right),\;\;\text{ for } n\geq 1.
\end{equation} 

c)
\begin{equation}\label{gf11}
c_{n}^{2} =\frac{R_3}{3}\left((n-1)\frac{a_n}{a_{n-2}}-(n-2)\right),\;\;\text{ for } n\geq 2.
\end{equation}

d.i)

\begin{eqnarray}\label{gfd11}
\frac{k+1}{n-k+1}a_{n-k}c_n^k &=&R_{k+1}\left(a_n-\frac{n-k}{n-k+1}a_{n-k}\right)
+\sum_{l=1}^{k-2}R_{k-l}\left(\frac{n+2}{n-l+1} \,c_n^l-\frac{n-k+l+1}{n-k+l+2}\,c_{n-k+l+1}^{l}\right)\nonumber\\
&&-\sum_{l=1}^{k-2}\frac{l+1}{n-l+1}c_{n}^{l}c_{n-l}^{k-l-1}-\sum_{l=1}^{k-2}\frac{R_{l+1}R_{k-l}}{n-l+1},\;\;\; 3\leq k\leq d,\;\;n\geq k.
\end{eqnarray}

d.ii)
\begin{eqnarray}\label{gfd12}
&& R_{k+1}\left(a_n-\frac{n-k}{n-k+1}a_{n-k}\right)
+\sum_{l=1}^{d}R_{k-l}\left(\frac{n+2}{n-l+1} \,c_n^l-\frac{n-k+l+1}{n-k+l+2}\,c_{n-k+l+1}^{l}\right)\nonumber\\
&&-\sum_{l=k-1-d}^{d}\frac{l+1}{n-l+1}c_{n}^{l}c_{n-l}^{k-l-1}=\sum_{l=1}^{k-2}\frac{R_{l+1}R_{k-l}}{n-l+1},\;\;\; d+1\leq k\leq 2d+1,\;\;n\geq k.
\end{eqnarray}

d.iii)
\begin{eqnarray}\label{gfd13}
&& R_{k+1}\left(a_n-\frac{n-k}{n-k+1}a_{n-k}\right)
+\sum_{l=1}^{d}R_{k-l}\left(\frac{n+2}{n-l+1} \,c_n^l-\frac{n-k+l+1}{n-k+l+2}\,c_{n-k+l+1}^{l}\right)\nonumber\\
&&=\sum_{l=1}^{k-2}\frac{R_{l+1}R_{k-l}}{n-l+1},\;\;\; k\geq 2d+2,\;\;n\geq k.
\end{eqnarray}

\end{proposition}

\begin{proof}
By differentiating \eqref{dorth0} we get
\begin{equation}\label{gf2}
xP'_n(x)+P_n(x)=P'_{n+1}(x)+\sum_{l=0}^{d}\gamma_{n}^{l}P_{n-l}'(x).
\end{equation}
Then by making the operations $n\alpha_n
Eq\eqref{gf2}+Eq\eqref{gf1}$ and $Eq\eqref{gf1}-\alpha_n
Eq\eqref{gf2}$ we obtain, respectively,

\begin{eqnarray}\label{gf3}
(n+1)\alpha_nxP'_n(x)=n\alpha_n\left(P'_{n+1}(x)+\sum_{l=0}^{d}\gamma_{n}^{l}P_{n-l}'(x)\right)+\sum_{k=1}^{n-1}R_{k+1}\alpha_{n-k}P'
_{n-k}(x)
\end{eqnarray}
and
\begin{eqnarray}\label{gf4}
(n+1)\alpha_n
P_n(x)=\alpha_n\left(P'_{n+1}(x)+\sum_{l=0}^{d}\gamma_{n}^{l}P_{n-l}'(x)\right)-\sum_{k=1}^{n-1}R_{k+1}\alpha_{n-k}P'
_{n-k}(x).
\end{eqnarray}
Inserting  \eqref{gf1} in the left-hand side of the equation $Eq\eqref{gf4}$ multiplied by $x$  we obtain
\begin{equation}\label{xgf4}
(n+1)\alpha_n\left(P_{n+1}+\sum_{l=0}^{d}\gamma_{n}^{l}P_{n-l}(x)\right)=\alpha_n\left(xP'_{n+1}(x)+\sum_{l=0}^{d}\gamma_{n}^{l}xP_{n-l}'(x)\right)-\sum_{k=1}^{n-1}R_{k+1}\alpha_{n-k}xP_{n-k}'(x).
\end{equation} 
Using \eqref{gf4} and \eqref{gf3} respectively  in the left hand side and right hand side of \eqref{xgf4} we get

\begin{eqnarray}
&&\frac{n+1}{n+2}\alpha_n\left(P_{n+2}'(x)+\sum_{l=0}^{d}\gamma_{n+1}^{l}P_{n+1-l}'(x)\right)-\frac{n+1}{n+2}\frac{\alpha_n}{\alpha_{n+1}}\sum_{k=1}^{n}R_{k+1}\alpha_{n+1-k}P_{n+1-k}'(x)\nonumber\\
&&+(n+1)\alpha_n\sum_{l=0}^{d}\gamma_{n}^{l}\frac{1}{n+1-l}P_{n+1-l}'(x)
+(n+1)\alpha_n\sum_{l=0}^{d}\gamma_{n}^{l}\frac{1}{n+1-l}\sum_{l'=0}^{d}\gamma_{n-l}^{l'}P_{n-l-l'}'(x)\nonumber\\
&&-(n+1)\alpha_n\sum_{l=0}^{d}\gamma_{n}^{l}\frac{1}{n+1-l}\frac{1}{\alpha_{n-l}}\sum_{k=1}^{n-l-1}R_{k+1}\alpha_{n-l-k}P_{n-k-l}'(x)=
\frac{n+1}{n+2}\alpha_n\left(P_{n+2}'(x)+\sum_{l=0}^{d}\gamma_{n+1}^{l}P_{n+1-l}'(x)\right)\nonumber\\
&&+\frac{1}{n+2}\frac{\alpha_n}{\alpha_{n+1}}\sum_{k=1}^{n}R_{k+1}\alpha_{n+1-k}P_{n+1-k}'(x)+\alpha_n\sum_{l=0}^{d}\gamma_{n}^{l}\frac{n-l}{n+1-l}P_{n+1-l}'(x)\nonumber\\
&&+\alpha_n\sum_{l=0}^{d}\gamma_{n}^{l}\frac{n-l}{n+1-l}\sum_{l'=0}^{d}\gamma_{n-l}^{l'}P_{n-l-l'}'(x)+\alpha_n\sum_{l=0}^{d}\gamma_{n}^{l}\frac{1}{n+1-l}\frac{1}{\alpha_{n-l}}\sum_{k=1}^{n-l-1}R_{k+1}\alpha_{n-l-k}P_{n-l-k}'(x)\nonumber\\
&&-\sum_{k=1}^{n-1}R_{k+1}\frac{n-k}{n+1-k}\alpha_{n-k}P_{n+1-k}'(x)-\sum_{k=1}^{n-1}R_{k+1}\frac{n-k}{n+1-k}\alpha_{n-k}\sum_{l=0}^{d}\gamma_{n-k}^{l}P_{n-k-l}'(x)\nonumber\\
&&-\sum_{k=1}^{n-1}R_{k+1}\frac{1}{n+1-k}\sum_{k'=1}^{n-k-1}R_{k'+1}\alpha_{n-k-k'}P_{n-k-k'}'(x).
\end{eqnarray}
It follows that
\begin{eqnarray}\label{xgf5}
&&-\frac{\alpha_n}{\alpha_{n+1}}\sum_{k=1}^{n}R_{k+1}\alpha_{n+1-k}P_{n+1-k}'(x)+\alpha_n\sum_{l=0}^{d}\gamma_{n}^{l}\frac{l+1}{n+1-l}P_{n+1-l}'(x)+\alpha_n\sum_{l=0}^{d}\gamma_{n}^{l}\frac{l+1}{n+1-l}\sum_{l'=0}^{d}\gamma_{n-l}^{l'}P_{n-l-l'}'(x)\nonumber\\
&&-(n+2)\alpha_n\sum_{l=0}^{d}\gamma_{n}^{l}\frac{1}{n+1-l}\frac{1}{\alpha_{n-l}}\sum_{k=1}^{n-l-1}R_{k+1}\alpha_{n-l-k}P_{n-k-l}'(x)=-\sum_{k=1}^{n-1}R_{k+1}\frac{n-k}{n+1-k}\alpha_{n-k}P_{n+1-k}'(x)\nonumber\\
&&-\sum_{k=1}^{n-1}R_{k+1}\frac{n-k}{n+1-k}\alpha_{n-k}\sum_{l=0}^{d}\gamma_{n-k}^{l}P_{n-k-l}'(x)
-\sum_{k=1}^{n-1}R_{k+1}\frac{1}{n+1-k}\sum_{k'=1}^{n-k-1}R_{k'+1}\alpha_{n-k-k'}P_{n-k-k'}'(x)
\end{eqnarray}

a) By comparing the coefficients of $P_{n+1}'(x)$ in the both sides of \eqref{xgf5} we obtain 
$$\frac{1}{n+1}\alpha_n\gamma_{n}^{0}=0,\;\; \text{ for } n\geq0,$$
and then $$\gamma_{n}^{0}=0,\;\; \text{ for } n\geq0.$$

b) Equating the coefficients of $P'_{n}(x)$ in the both sides of the
 equation \eqref{xgf5} gives
 $$\frac{2}{n}\alpha_n\gamma_n^1=R_2\frac{\alpha_n}{\alpha_{n+1}}\alpha_n-R_2\frac{n-1}{n}\alpha_{n-1},\quad \text{ for } n\geq 1,$$
which can be written as
$$\gamma _n^1=na_n-(n-1)a_{n-1},\;\;\text{ for } n\geq 1.$$

Now by equating the coefficients of $P'_{n+1-k}(x)$ for $k\geq 2$  in the both sides of the equation \eqref{xgf5} we obtain

\begin{eqnarray}\label{dgf12}
&&\frac{k+1}{n+1-k}\alpha_n\gamma_{n}^{k}+R_{k+1}\alpha_{n-k}\frac{n-k}{n+1-k}-R_{k+1}\frac{\alpha_n}{\alpha_{n+1}}\alpha_{n+1-k}+\alpha_n\sum_{l=1}^{d}\gamma_{n}^{l}\frac{l+1}{n+1-l}\sum_{l'=1}^{d}\gamma_{n-l}^{l'}\delta_{l+l'}^{k-1}\nonumber\\
&&-(n+2)\alpha_n\sum_{l=1}^{d}\gamma_{n}^{l}\frac{1}{n+1-l}\sum_{k'=1}^{n-l-1}R_{k'+1}\frac{\alpha_{n-l-k'}}{\alpha_{n-l}}\delta_{l+k'}^{k-1}+\sum_{k'=1}^{n}R_{k'+1}\alpha_{n-k'}\frac{n-k'}{n+1-k'}\sum_{l=1}^{d}\gamma_{n-k'}^{l}\delta_{k'+l}^{k-1}\nonumber\\
&&+\sum_{k''=1}^{n}R_{k''+1}\frac{1}{n+1-k''}\sum_{k'=1}^{n-k''}R_{k'+1}\alpha_{n-k'-k''}\delta_{k'+k''}^{k-1}=0,\;\;\text{for }  n\geq k.
\end{eqnarray}
 
then by taking $k=2$ in \eqref{dgf12} we retrieve c) and by considering $3\leq k\leq d$, $d+1\leq k\leq 2d+1$ and $k\geq 2d+2$ we obtain d.i.), d.ii.) and d.iii.) respectively.
 
\end{proof}

In the following corollaries we adopt the same conditions and notations as in Proposition \ref{prop2}.
\begin{corollary}\label{cor1}
If $R_2=R_3=\cdots=R_{d+1}=0$ and $\alpha_n\neq 0$ for $n\geq 1$, then $R(t)=0$, $F(t)$ is arbitrary and $F(xt)=\sum_{n\geq 0}\alpha_nx^nt^n$
generates the monomials $\{x^n\}_{n\geq 0}$.
\end{corollary}

\begin{proof}
As $R_1=R_2=\cdots=R_{d+1}=0$, it is enough to show by induction that $R_n=0$ for $n\geq d+2$.
For $n=1,2,...,d+2$, the equation \eqref{gf1} gives $P_1(x)=x$, $P_2(x)=x^2,...,P_{d+1}(x)=x^{d+1}$ and $P_{d+2}(0)=\frac{-R_{d+2}\alpha_1}{(d+2)\alpha_{d+2}}$. But according to equation \eqref{dorth0}, for $n=d+1$,  $P_{d+2}(0)=0$ and then $R_{d+2}=0$.

 Now assume that $R_k=0$ for $d+2\leq k\leq n-1$. According to \eqref{gf1} we have, for $d+2\leq k\leq n-1$, $P_k(0)=0$ and $P_n(0)=-\frac{R_n\alpha_1}{n\alpha_n}$.
On other hand, by the shift $n\rightarrow n-1$ in \eqref{dorth0} we have $P_n(0)=0$ and thus $R_n=0$. As $R(t)=0$, the generating function \eqref{gf00} reduces to  $F(xt)=\sum_{n\geq 0}\alpha_nx^nt^n$ which generates the monomials with $F(t)$ arbitrary.
\end{proof}

\begin{corollary}\label{cor2}
If $R_{d+2}=R_{d+3}=\cdots=R_{2d+2}=0$ then $R(t)=R_2t^2/2+R_3t^3/3+\cdots+R_{d+1}t^{d+1}/(d+1)$. 
  \end{corollary}
\begin{proof}
We will use \eqref{gfd13} and proceed by induction on
$k$ to show that $R_k=0$ for $k\geq 2d+3$. Indeed $k=2d+2$ and $n=2d+2$
in \eqref{gfd13} leads to $a_{2d+2}R_{2d+3}=0$ and since $a_n\neq 0$  we get $R_{2d+3}=0$.
Suppose that $R_{2d+3}=R_{2d+4}=\cdots=R_{k}=0$, then for $n=k$ the
equation \eqref{gfd13} gives $a_{k}R_{k+1}=0$ and finally $R_{k+1}=0$.
\end{proof}

\begin{corollary}\label{cor03}
If $R_{\kappa+d+1}=\cdots=R_{\kappa+1}=R_{\kappa}=R_{\kappa-1}=\cdots=R_{\kappa-d}=0$ for some $\kappa\geq 3d+3$, then
$R_{d+2}=R_{d+3}=\cdots=R_{2d+2}=0$.
\end{corollary}

\begin{proof}

$\bullet$ Let $k=\kappa$ in \eqref{gfd13}, then for $n\geq \kappa$ the
fraction $\sum_{l=1}^{\kappa-2}\frac{R_{l+1}R_{\kappa-l}}{n-l+1}$, as function of integer $n$, is
null even for real $n$. So, 
\begin{equation}
\lim_{x\to l-1}(x-l+1)\sum_{s=1}^{\kappa-2}\frac{R_{s+1}R_{\kappa-s}}{x-s+1}=R_{l+1}R_{\kappa-l}=0,\;\text{for}\; 1\leq l \leq \kappa-2
\end{equation}
which is $R_{d+2}R_{\kappa-d-1}=0$ when $l=d+1$. Supposing  $R_{d+2}\neq 0$ leads to
$R_{\kappa-d-1}=0$. So $R_{\kappa+d}=R_{\kappa+d-1}=\cdots=R_{\kappa-d}=R_{\kappa-d-1}=0$ and with the
same procedure we find $R_{\kappa-d-2}=0$. Going so on till we arrive at $R_{d+2}=0$ which contradicts $R_{d+2}\neq 0$.

$\bullet$  By taking successively $k=\kappa +r, \kappa +r-1,...,\kappa$  in  \eqref{gfd13}, for $1\leq r\leq d$, we find
\begin{eqnarray*}
&&R_{l+1}R_{\kappa+r-l}=0 \text { for } 1\leq l \leq \kappa+r-2,\\
&&R_{l+1}R_{\kappa+r-1-l}=0 \text{ for }  1\leq l \leq \kappa+r-3,\\
&&\qquad\qquad\vdots\\
&&R_{l+1}R_{\kappa-l}=0 \text{ for } 1\leq l \leq \kappa-2.
\end{eqnarray*}
If $R_{d+2+r}\neq 0$ then by taking $l=d+1+r$ we get $R_{\kappa-d-1}=R_{\kappa-d-2}=\cdots=R_{\kappa-d-r-1}=0$. So $R_{\kappa+d-r}=R_{\kappa+d-r-1}=\cdots=R_{\kappa-d-r-1}=0$ and with the
same procedure we find $R_{\kappa-d-r-2}=R_{\kappa-d-r-3}=\cdots=R_{\kappa-d-2r-2}=0$.  Going so on till we arrive at $R_{d+2+r}=0$ which contradicts $R_{d+2+r}\neq 0$. 
\end{proof}

\begin{corollary}\label{cor5}
 If $a_n$ is a rational function of $n$ then $R_{d+2}=R_{d+3}=\cdots=R_{2d+2}=0$.
\end{corollary}

\begin{proof} 
From \eqref{cgf10}, \eqref{gf11} and \eqref{gfd11} observe that $c_n^l$ will also be a rational function of $n$. Then it follows that, in \eqref{gfd13}, two fractions are equal for natural numbers $n\geq k$, $k\geq 2d+2$, and consequently will be for real numbers $n$. If we denote by $N_s(G(x))$ the number of singularities of a rational function
$G(x)$ then we can easily verify, for all rational functions  $G$ and $\tilde{ G}$ of $x$ and a constant $a\neq 0$, that:

a) $N_s(G(x+a))=N_s(G(x))$,

b) $N_s(aG(x))=N_s(G(x))$,

c) $N_s(G(x)+\tilde{ G}(x))\leq N_s(G(x))+N_s(\tilde{ G}(x))$.

Using  property a) of $N_s$ we have $$N_s\left(\frac{n-k}{n-k+1}a_{n-k}\right)=N_s\left(\frac{n}{n+1}a_{n}\right) \text{ and } N_s\left(\frac{n-k+l+1}{n-k+l+2}\,c_{n-k+l+1}^{l}\right)=N_s\left(\frac{n}{n+1} \,c_n^l\right).$$ 
According to properties b) and c) of $N_s$, the $N_s$ of the left-hand side of \eqref{gfd12} is finite and independent of $k$. Thus, the right-hand side of \eqref{gfd12} has a finite number  of
singularities which is independent of $k$. As consequence there exists a $k_1\geq 3d+3$ for which $R_{l+1}R_{k-l}=0$
 for all $k \geq k_1-d-1$  and $k_1-d-1\leq l\leq k$. According to Corollary~\ref{cor1}, there exists a $k_0$ such that  $2\leq k_0\leq d+1$ and  $R_{k_0}\neq 0$. So, taking successively  
 $k=k_0+l$ with $l=k_1+d,k_1+d-1,...,k_1-d-1$ we get
$R_{k_1+d+1}=R_{k_1+d}=\cdots=R_{k_1-d}=0$. Then, by Corollary~\ref{cor03} we have  $R_{d+2}=R_{d+3}=\cdots=R_{2d+2}=0$.

\end{proof}

The fact that $a_n$ is a rational function of $n$ means that $F(\epsilon z)=\sum_{n\geq
0}\alpha_n(\epsilon z)^n$ (where $\epsilon$ is the quotient of the leading coefficients of the numerator and the denominator of $a_n$) is a generalized  hypergeometric series, i.e. of the form:

\begin{eqnarray}\label{F01}
{}_{p}F_{q}\left(
\begin{array}{llll}
\left(\mu_l\right)_{l=1}^p\\
\left(\nu_l\right)_{l=1}^q
\end{array};z\right)={}_{p}F_{q}\left(
\begin{array}{llll}
\mu_1,\mu_2,&...,&\mu_p\\
\nu_1,\nu_2,&...,&\nu_q
\end{array};z\right)=\sum_{n\geq 0}\frac{(\mu_1)_n(\mu_2)_n\cdots(\mu_p)_n}{(\nu_1)_n(\nu_2)_n\cdots(\nu_q)_n}\frac{z^n}{n!}
\end{eqnarray}
where $\left(\mu_l\right)_{l=k}^p$ denotes the array of complex parameters $\mu_k,\mu_{k+1},...,\mu_p$, and if $k>p$ we take the convention that $\left(\mu_l\right)_{l=k}^p$ is the empty array. The symbol $(\mu)_n$ stands for the shifted factorials, i.e. 
\begin{equation}
(\mu)_0=1,\;\;(\mu)_n=\mu(\mu+1)\cdots(\mu+n-1),\;\;n\geq 1. 
\end{equation}

As an interesting consequence, from Corollary~\ref{cor5} and Corollary~\ref{cor2} we state the following result, which can be interpreted as a generalization of the Appell case in the above Theorem 3 (see also \cite{Douak96}):

\begin{theorem}\label{th4}
Let  $\{P_n\}$ be a $d$-PS generated by \eqref{gf00} with $F(z)$ a generalized  hypergeometric series. Then $R(t)=R_2t^2/2+R_3t^3/3+\cdots+R_{d+1}t^{d+1}/(d+1)$. 
\end{theorem} 

\begin{proof}
$F(z)=\sum_{n\geq 0}\alpha_nz^n$ has the form \eqref{F01}. Then $a_n=\alpha_n/\alpha_{n+1}$ is a rational function of $n$, since $(\mu)_{n+1}/(\mu)_{n}=n+\mu$. The use of Corollary~\ref{cor5} and Corollary~\ref{cor2} completes the proof.
\end{proof}
\begin{corollary}\label{cor06}
Let $R(t)=R_2t^2/2+R_3t^3/3+\cdots+R_{d+1}t^{d+1}/(d+1)$. Then,\\
i) If $R_{d+1}=0$ we have $c_{n+d}^{d}c_{n}^{d}=0$, for $n\geq d+1$.\\
ii) If $R_{d+1}\neq 0$ then
\begin{equation}\label{gf111a}
c_{n}^{d} =\frac{R_{d+1}}{d+1}\left((n+1)\frac{b_{n-d}}{b_n}-(n-d)\right),\;\;\text{ for } n\geq d+1,
\end{equation} 
where $b_{md+r}=(b_{d+r}-b_{r})m+b_{r},\;for\; m\geq 0,\; 1\leq r\leq d$.\\
iii) The $\{c_{n}^{m}\}_{1 \leq m \leq d-1}$ can be calculated recursively by solving the following $d$-order linear difference equations:
\begin{eqnarray}\label{gfd122}
&&\frac{1}{n-m+1}\left(R_{d+1}(n+2)-(n-d)c_{n-m}^{d}\right)\,c_n^{m}-\frac{1}{n-d+1}\left(R_{d+1}(n-d)+(d+1)c_{n}^{d}\right)\,c_{n-d}^{m}+\nonumber\\
&&+\sum_{l=m+1}^{d}R_{m+d+1-l}\left(\frac{n+2}{n-l+1} \,c_n^l-\frac{n-m-d+l}{n-m-d+l+1}\,c_{n-m-d+l}^{l}\right)
-\sum_{l=m+1}^{d-1}\frac{l+1}{n-l+1}c_{n}^{l}c_{n-l}^{m+d-l}\nonumber\\
&&=\sum_{l=m}^{d-1}\frac{R_{l+1}R_{m+d+1-l}}{n-l+1},\;\;\; 1\leq m\leq d-1,\;\;n\geq m+d+1.
\end{eqnarray}
\end{corollary}

\begin{proof}\\
\textbf{The proof of i)} \\
Put $k=2d+1$ in \eqref{gfd12} to get the following Riccati equation for $\{c_{n}^{d}\}$:
\begin{equation}\label{gfd123}
R_{d+1}\left((n+2)\,c_n^d-(n-d)\,c_{n-d}^{d}\right)-(d+1)c_{n}^{d}c_{n-d}^{d}-R_{d+1}^{2}=0,\;\;for\; n\geq 2d+1.
\end{equation}
By taking $R_{d+1}=0$ in \eqref{gfd123}, i) follows immediately.\\
\textbf{The proof of ii)} \\
Substituting \eqref{gf111a} in \eqref{gfd123} we find the $2d$-linear homogeneous equation
\begin{equation}\label{gfd124}
b_n-2b_{n-d}+b_{n-2d}=0,\;\;for\; n\geq 2d+1.
\end{equation}
By writing $n=md+r$, where $m,r$ are natural numbers with $1\leq r \leq d$, the equation \eqref{gfd124} can be solved by summing twice to find that\\
 $$b_{md+r}=(b_{d+r}-b_{r})m+b_{r},\;for\; m\geq 0,\; 1\leq r\leq d.$$
\textbf{The proof of iii)} \\
Since $d+1\leq k\leq 2d+1$ we have $R_{k+1}=0$ and $R_{k-l}=0$ for $l\leq k-2-d$. So, we can write \eqref{gfd12} as
\begin{eqnarray}\label{gfd125}
&&\sum_{l=k-d}^{d}R_{k-l}\left(\frac{n+2}{n-l+1} \,c_n^l-\frac{n-k+l+1}{n-k+l+2}\,c_{n-k+l+1}^{l}\right)+R_{d+1}\left(\frac{n+2}{n-k+d+2} \,c_n^{k-1-d}-\frac{n-d}{n-d+1}\,c_{n-d}^{k-1-d}\right)\nonumber\\
&&-\sum_{l=k-d}^{d-1}\frac{l+1}{n-l+1}c_{n}^{l}c_{n-l}^{k-l-1}-\frac{k-d}{n-k+d+2}c_{n}^{k-1-d}c_{n-k+1+d}^{d}-\frac{d+1}{n-d+1}c_{n}^{d}c_{n-d}^{k-1-d}=\nonumber\\
&&=\sum_{l=k-1-d}^{k-2}\frac{R_{l+1}R_{k-l}}{n-l+1},\;\;\; d+1\leq k\leq 2d+1,\;\;n\geq k.
\end{eqnarray}
Putting $m=k-d-1 \neq 0$ in \eqref{gfd125} and rearranging we obtain \eqref{gfd122}.
\end{proof}

\begin{remark}
In the case of $d$-OPSs, in Corollary~\ref{cor06} the polynomial $R(t)$ is of degree $d+1$. Otherwise (i.e. $R_{d+1}=0$), we have a contradiction with the regularity conditions $\gamma_{n}^d\neq 0$, for $n\geq d$.
\end{remark}

\begin{corollary}\label{cor02}
The $d$-PS is classical if and only if $R(t)=R_2t^2/2+R_3t^3/3+\cdots+R_{d+1}t^{d+1}/(d+1)$ with $R_{d+1}\neq 0$. 
\end{corollary}

\begin{proof}\text{  }\\
1) Assume that the $d$-PS is classical. From \eqref{gf3} and Definition~\ref{def2}  we have $R_{k+1}=0$ for $d+1\leq k\leq n-1$. We get $R(t)$ by taking $n\geq 2d+2$ and using Corollary~\ref{cor2}. Now we show that $R_{d+1}\neq 0$. Equation \eqref{gf3} becomes
\begin{equation}\label{clas1}
xQ_n(x)=Q_{n+1}(x)+\sum_{l=1}^{d}\tilde{\gamma}_{n}^{l}Q_{n-l}(x),\quad n\geq 0,
\end{equation}
where $Q_n(x)=(n+1)^{-1}P'_{n+1}(x)$ and
\begin{equation}\label{clas2}
\tilde{\gamma}_{n}^{l}=\frac{n+1-l}{n+2}\left(\gamma_{n+1}^{l}+\frac{R_{l+1}\alpha_{n+l-1}}{(n+1)\alpha_{n+1}}\right), \quad n\geq d.
\end{equation}
From \eqref{gf111a}, if $R_{d+1}=0$ then $\gamma_{n+1}^{d}=0$, and \eqref{clas2} gives $\tilde{\gamma}_{n}^{d}=0$, for $n\geq d$. So, $\{Q_n\}$ is not a $d$-PS which contradicts the fact that $\{P_n\}$ is classical (see Definition~\ref{def2}).

2) Assume that $R(t)=R_2t^2/2+R_3t^3/3+\cdots+R_{d+1}t^{d+1}/(d+1)$ with $R_{d+1}\neq 0$ , then the PS of the derivatives $\{Q_n\}$ satisfy \eqref{clas1}  and are generated by $F'(xt-R(t))=\sum_{n\geq 0}(n+1)\alpha_{n+1}Q_n(x)t^n$. Using  Corollary~\ref{cor06} we find that  $\tilde{c}_{n}^{d}:=(n+1)\alpha_{n+1}\tilde{\gamma}_{n}^{d}/((n-d+1)\alpha_{n-d+1})$ satisfies \eqref{gfd123}. And  according to the same expression \eqref{gfd123}, we should have,  if $c_{n}^{d}=0$ or $\gamma_{n}^{d}=0$ (for $n\geq d+1$), $R_{d+1}=0$. Therefore,  there exists for $\tilde{c}_{n}^{d}$, since $R_{d+1}\neq 0$, a $n_0\geq d+1$  such that $\tilde{c}_{n_0}^{d}\neq 0$ or $\tilde{\gamma}_{n_0}^{d}\neq 0$. This means that $\{P_n\}$ is classical. 

\end{proof}

\section{The $d$-symmetric case}

The main result of this section is the following:
\begin{theorem}\label{Th5}
If $\{P_n\}$ is a $d$-symmetric $d$-PS generated by \eqref{gf00} then $R(t)=R_{d+1}t^{d+1}/(d+1)$.
\end{theorem}
Theorem~\ref{Th5} generalizes Theorem~\ref{thmm0} and Theorem~\ref{thmm1} mentioned above. Its proof is quite similar to that of Theorem~\ref{thmm0} in \cite{meskzahaf} and it requires the following Lemmas.
\begin{lemma}\label{Lem1}
If $\{P_n\}$ is a $d$-symmetric $d$-PS generated by \eqref{gf00} then
\begin{equation}\label{dsym0}
R(t)=\sum_{k\geq 1}\frac{R_{k(d+1)}}{k(d+1)}t^{k(d+1)}. 
\end{equation}
\end{lemma}
\begin{proof}
Let $\{P_n\}$ be a $d$-symmetric $d$-PS satisfying \eqref{dorth0} and generated by \eqref{gf00}.
Then it has, according to Definition~\ref{def3}, the property
\begin{equation}\label{dsym}
P_n(\omega x)=\omega ^nP_n(x),
\end{equation}
where $\omega=\exp(2\pi i/(d+1))$. It follows that   \eqref{dorth0} becomes \cite{MaroniDouak}
\begin{eqnarray}
\left\{
\begin{array}{l}
xP_n(x)=P_{n+1}(x)+\gamma_n^d P_{n-d}(x),\quad n\geq
0,\\
P_{-n}(x)=0,\;\;1\leq n\leq d,\;\;\text{and }P_0(x)=1.
\end{array}
\right.
\end{eqnarray} 
Let us show that $R_k=0$ when $k$ is not a multiple of $d+1$. First  we replace $x$ by $\omega x$ in \eqref{gf1} and use \eqref{dsym} with $P'_n(\omega x)=\omega^{n-1}P'_n(x) $ to get 
\begin{equation}\label{gfdsym} 
\alpha_n xP'_n(x)-\sum_{k=1}^{n-1}R_{k+1}\alpha_{n-k}\omega^{-k-1}P'_{n-k}(x)=n\alpha_nP_n(x), \;\; n\geq 2.
\end{equation}
Subtracting \eqref{gfdsym} from \eqref{gf1} gives
\begin{equation}
\sum_{k=1}^{n-1}R_{k+1}\alpha_{n-k}(1-\omega^{-k-1})P'_{n-k}(x)=0,\;\;n\geq 2
\end{equation}
which leads to
$$R_{k}\alpha_{n-k+1}(1-\omega^{-k})=0,\;\; \text{for }2\leq k\leq n,\;\;n\geq 2.$$
Since  $\omega^{k}\neq 1$, provided $k$ is not a multiple of $d+1$, gives the result.
\end{proof}

By Lemma~\ref{Lem1} and putting $T_k=R_{k(d+1)}$ for $k\geq 0$, the equations in Proposition~\ref{prop2} simplify to particular forms. Indeed, from  \eqref{cgf10}, \eqref{gf11} and \eqref{gfd11} we get
\begin{equation}\label{cdn}
c_n^d=\frac{T_{1}}{d+1}\left((n-d+1)\frac{a_n}{a_{n-d}}-(n-d)\right),\;\;\text{for } n\geq d.
\end{equation}
The equation \eqref{gfd12}, with $k=2d+1$, becomes
\begin{eqnarray}\label{gfdsym12}
&& T_{2}\left(a_n-\frac{n-2d-1}{n-2d}a_{n-2d-1}\right)
+T_{1}\left(\frac{n+2}{n-d+1} \,c_n^d-\frac{n-d}{n-d+1}\,c_{n-d}^{d}\right)\nonumber\\
&&-\frac{d+1}{n-d+1}c_{n}^{d}c_{n-d}^{d}=\frac{T_{1}^2}{n-d+1},\;\;n\geq 2d+1,
\end{eqnarray}
which by \eqref{cdn} takes the form
\begin{equation}\label{gf1102}
\frac{(d+1)T_{2}}{T_{1}^2}\left(1-\frac{n-2d-1}{n-2d}\frac{a_{n-2d-1}}{a_{n}}\right)=\frac{n+1}{a_n}-\frac{2(n-d+1)}{a_{n-d}}+\frac{n-2d+1}{a_{n-2d}},\;\;\text{
for } n\geq 2d+1.
\end{equation}
Finally, the equation \eqref{gfd13} simplifies to
\begin{eqnarray}\label{gfd1313}
&& T_{k+1}\left(a_n-\frac{n-k(d+1)-d}{n-k(d+1)-d+1}a_{n-k(d+1)-d}\right)
+T_{k}\left(\frac{n+2}{n-d+1} \,c_n^d-\frac{n-k(d+1)+1}{n-k(d+1)+2}\,c_{n-k(d+1)+1}^{d}\right)\nonumber\\
&&=\sum_{l=0}^{k-1}\frac{T_{l+1}T_{k-l}}{n-l(d+1)-d+1},\;\;\; k\geq 2,\;\;n\geq k(d+1)+d.
\end{eqnarray}
This equation will be denoted by $E_k(n)$ in below.

\begin{lemma}\label{Lem3}
If  $T_{2}=0$ then  $R(t)=T_{1}t^{d+1}/(d+1)$.
\end{lemma}

\begin{proof}
According to Corollary \ref{cor2}, if $T_2=R_{2(d+1)}=0$ then $R(t)=T_{1}t^{d+1}/(d+1)$, since in this case we have $R_{d+2}=R_{d+3}=\cdots=R_{2d+1}=0$.
\end{proof}

\begin{lemma}\label{cor3}
If $T_{m}=T_{m+1}=0$ for some $m\geq 3$, then
$T_2=0$.
\end{lemma}

\begin{proof}
$T_{m}=T_{m+1}=0$ means that $R_{(d+1)m}=R_{(d+1)(m+1)}=0$. Also by Lemma~\ref{Lem1}, we have $R_{(d+1)m+d}=R_{(d+1)m+d-1}=\cdots=R_{(d+1)m+1}=0$ and $R_{(d+1)m-1}=R_{(d+1)m-2}=\cdots=R_{(d+1)m-d}=0$ which represents the condition of corollary~\ref{cor03} with $\kappa=(d+1)m\geq 3(d+1)$ and therefore gives $R_{2d+2}=T_2=0$.
\end{proof}

\begin{lemma}\label{cor6}
If $T_{\kappa}=T_m=0$ for some $\kappa\neq m\geq 3$, then
$T_2=0$.
\end{lemma}

\begin{proof}
The proof is similar to that of Corollary~7  in \cite{meskzahaf}. Let assume that $T_{\kappa+1}\neq 0$ and $T_{m+1}\neq 0$, since if not, we apply Corollary~\ref{cor3}. When $m> \kappa$ and by using \eqref{gfd1313}, the following operations
\begin{eqnarray}
&&\left[E_{\kappa}(n+(d+1)m+d)/T_{\kappa+1}-E_{m}(n+(d+1)m+d)/T_{m+1}\right]\nonumber\\
&&-\left[E_{\kappa}(n)/T_{\kappa+1}-E_{m}(n+(d+1)(m-\kappa)+d)/T_{m+1}\right]
\end{eqnarray}   
give
\begin{equation}
\left(\frac{n}{n+1}-\frac{n+(d+1)(m-\kappa)}{n+(d+1)(m-\kappa)+1}\right)a_n=Q(n)
\end{equation}
where $Q(n)$ is a rational function of $n$. Consequently, $a_n$ is a rational function of $n$ and by Corollary~\ref{cor5} we have $T_2=0$.
\end{proof}

\begin{lemma}\label{cor7}
The following equality is true for $k\geq 3\hbox{ and } n\geq k(d+1)+2d+1.$
\begin{eqnarray}\label{gfd23}
T_{k-1}D_{k+1}(a_n-\tilde{
a}_{n-k(d+1)-2d-1})-T_{k+1}D_{k}(a_{n-d-1}-\tilde{
a}_{n-k(d+1)-d})=\sum_{l=1}^{k-1}\frac{V_{k,l}}{n-l(d+1)-d+1},
\end{eqnarray}

where
\begin{itemize}
\item $D_{k,l}=T_kT_{k-l+1}-T_{k+1}T_{k-l}$.
\item $D_{k}=D_{k,1}=T_k^2-T_{k+1}T_{k-1}$.
\item
$V_{k,l}=\frac{T_1}{2}\left(T_lT_{k+1}D_{k-1,l-1}-T_{l+1}T_{k-1}D_{k,l}\right)$.
\item $\tilde{ a}_n=\frac{n}{n+1}a_n$.
\end{itemize}
\end{lemma}

\begin{proof}
Just by making  the following combinations it is  easy to get \eqref{gfd23}:
\begin{equation*}
T_{k+1}\left(T_{k-1}E_k(n)-T_{k}E_{k-1}(n-d-1)\right)-T_{k-1}\left(T_kE_{k+1}(n)-T_{k+1}E_k(n-d-1)\right).
\end{equation*}
\end{proof}
 To prove Theorem~\ref{Th5} it is sufficient, according to Lemma~\ref{Lem3}, to show that $T_2=0$. To this end, we will consider three cases:\\

{\bf\textit { Case 1:}} There exists $k_0\geq 3$ such that $D_k\neq 0$ for $k\geq k_0$.\\

Considering Lemma~\ref{cor6}, we can choose $\tilde{k}\geq k_0$ such that $T_{k}\neq 0$ for $k\geq \tilde{k}-1$. Let define, for $k\geq \tilde{k}$, $\bar{D}_{k}=\frac{D_{k}}{T_{k-1}T_{k}}$ and $\bar {E}_k(n)$ be the equation \eqref{gfd23} divided by $T_{k-1}T_{k}T_{k+1}$. By making the operations
\begin{equation*}
\left[\bar{D}_{k-1}\bar{E}_k(n+d+1)-\bar{D}_k\bar{E}_{k-1}(n)\right]+\left[\bar{D}_{k+1}\bar{E}_{k-1}(n-d-1)-\bar{D}_{k}\bar{E}_k(n)\right]
\end{equation*}
we get, for $k\geq \tilde{k}+1$, the equation
\begin{equation}\label{Dk101}
{a}_{n+d+1}-{a}_{n-2(d+1)}-\tilde{D}_k({a}_n-{a}_{n-d-1})=\sum_{l=1}^{k}\frac{{W}_{k,l}}{n-(d+1)l+2}:=Q_k^{(1)}(n),
\end{equation}
where $W_{k,l}$ is independent of $n$ and

$$\tilde{D}_k=\frac{\bar{D}_{k}^2+\bar{D}_{k}\bar{D}_{k-1}+\bar{D}_{k}\bar{D}_{k+1}}{\bar{D}_{k-1}\bar{D}_{k+1}}.$$ 
Similarly, by  the operations
\begin{equation}\label{Dk2.1}
\left[\bar{D}_{k}\bar{E}(k,n+d+1)-\bar{D}_{k+1}\bar{E}(k-1,n+d+1)\right]+\left[\bar{D}_{k}\bar{E}(k-1,n)-\bar{D}_{k-1}\bar{E}(k,n)\right]
\end{equation}
and the shift $n\rightarrow n+(d+1)k-1$ in \eqref{Dk2.1} we
obtain
\begin{equation}\label{Dk3}
\tilde{a}_{n+d+1}-\tilde{a}_{n-2(d+1)}-\tilde{D}_k(\tilde{a}_n-\tilde{a}_{n-d-1})=\sum_{l=1}^{k}\frac{\widetilde{W}_{k,l}}{n+(d+1)l-d}:=\widetilde{Q}_k^{(1)}(n),
\end{equation}
where $\widetilde{W}_{k,l}$ is independent of $n$.
Now, for $k\neq \kappa\geq \tilde{k}+1$, the equations
\eqref{Dk101} and \eqref{Dk3} give, respectively,
\begin{equation}\label{Dk2}
(\tilde{D}_{\kappa}-\tilde{D}_{k})(a_n-a_{n-d-1})=Q_k^{(1)}(n)-Q_{\kappa}^{(1)}(n)
\end{equation}
and 
\begin{equation}\label{Dk4}
(\tilde{D}_{\kappa}-\tilde{D}_k)\left(\frac{n}{n+1}a_n-\frac{n-d-1}{n-d}a_{n-d-1}\right)=\widetilde{Q}_k^{(1)}(n)-\widetilde{Q}_{\kappa}^{(1)}(n).
\end{equation}

If $\tilde{D}_k\neq \tilde{D}_{\kappa}$ for some $k\neq \kappa\geq
\tilde{k}+1$, then by \eqref{Dk2} and \eqref{Dk4} we can
eliminate $a_{n-d-1}$ to get that $a_n$ is a rational function of
$n$. So, by Corollary~\ref{cor5}, we have $T_2=0$.

If $\tilde{D}_k= D$ for $k\geq \tilde{k}+1$, then
\eqref{Dk101} and \eqref{Dk3} become, respectively,
\begin{equation}\label{Dk6}
a_{n+d+1}-a_{n-2(d+1)}-D(a_n-a_{n-d-1})=Q_k^{(1)}(n)
\end{equation}
and
\begin{equation}\label{Dk66}
\frac{n+d+1}{n+d+2}a_{n+d+1}-\frac{n-2d-2}{n-2d-1}a_{n-2(d+1)}-D\left(\frac{n}{n+1}a_n-\frac{n-d-1}{n-d}a_{n-d-1}\right)=\widetilde{Q}_k^{(1)}(n).
\end{equation}
The combinations
$\left((n+d+2)Eq\eqref{Dk66}-(n+d+1)Eq\eqref{Dk6}\right)/(d+1)$ and
$((n-2d-1)Eq\eqref{Dk66}-(n-2d-2)Eq\eqref{Dk6})/(d+1)$ give, respectively,
\begin{equation}\label{Dk8}
\frac{3a_{n-2d-2}}{n-2d-1}-D\left(-\frac{a_n}{n+1}+\frac{2a_{n-d-1}}{n-d}\right)=Q_k^{(3)}(n)
\end{equation}
and
\begin{equation}\label{Dk9}
\frac{3a_{n+d+1}}{n+d+2}-D\left(\frac{2a_n}{n+1}-\frac{a_{n-d-1}}{n-d}\right)=Q_k^{(4)}(n).
\end{equation}
By shifting $n\rightarrow n+d+1$ in \eqref{Dk8} we obtain
\begin{equation}\label{Dk10}
\frac{3a_{n-d-1}}{n-d}-D\left(-\frac{a_{n+d+1}}{n+d+2}+\frac{2a_{n}}{n+1}\right)=Q_k^{(3)}(n+d+1).
\end{equation}
The coefficients $a_{n+d+1}$ and $a_{n-d-1}$ can be eliminated by the operations 
$D\times Eq\eqref{Dk9}-3\,Eq\eqref{Dk10}$ and $3\,Eq\eqref{Dk9}-D\times Eq\eqref{Dk10}$ leaving us with
\begin{equation}\label{Dk11}
\frac{6D-2D^2}{n+1}a_n+\frac{D^2-9}{n-d}a_{n-d-1}=Q_k^{(5)}(n)
\end{equation} 
and 
\begin{equation}\label{Dk12}
\frac{9-D^2}{n+d+2}a_{n+d+1}-\frac{6D-2D^2}{n+1}a_{n}=Q_k^{(6)}(n).
\end{equation} 

Finally, the shifting $n\rightarrow n-d-1$ in \eqref{Dk12} leads to
\begin{equation}\label{Dk13}
\frac{9-D^2}{n+1}a_{n}-\frac{6D-2D^2}{n-d}a_{n-d-1}=Q_k^{(6)}(n-d-1)
\end{equation}
and the operation $(6D-2D^2)Eq\eqref{Dk11}+(D^2-9)Eq\eqref{Dk13}$ gives
\begin{equation}\label{Dk14}
[(6D-2D^2)^2+(D^2-9)^2]a_n=Q_k^{(7)}(n).
\end{equation}

According to manipulations made above, $Q_k^{(7)}(n)$ is a
rational function of $n$. As consequence, if $D\neq 3$, $a_n$ is a rational function of $n$ and then $T_2=0$.

Now, we explore the case $D=3$. According to the left-hand sides of \eqref{Dk9} and \eqref{Dk10}, we have%
\begin{equation*}
Q_{k}^{(3)}\left( n+d+1\right) =Q_{k}^{(4)}\left( n\right),
\end{equation*}%
which can be written  as  
\begin{equation}\label{Eq77}
(n+2d+2)Q_{k}^{(1)}\left( n+d+1\right)-(n-2d-2)Q_{k}^{(1)}\left(
n\right) =\left( n+2d+3\right)
\widetilde{Q}_{k}^{(1)}\left( n+d+1\right)-\left( n-2d-1\right)
\widetilde{Q}_{k}^{(1)}\left( n\right).
\end{equation}%
By using, from \eqref{Dk101} and \eqref{Dk3}, the expressions of $Q_{k}^{(1)}(n)$ and $\widetilde{Q}_{k}^{(1)}(n)$ with $W_{k,k+1}=\widetilde{W}_{k,k+1}=W_{k,0}=\widetilde{W}_{k,0}=0$ we obtain

\begin{eqnarray}\label{Eq7a}
\sum_{l=0}^{k}\frac{((d+1)l+2d)W_{k,l+1}-((d+1)l-2d-4)W_{k,l}}{n-(d+1)l+2}=\sum_{l=0}^{k}\frac{(d+1)((2-l)\widetilde{W}_{k,l}-(l+2)\widetilde{W}_{k,l+1})}{n+(d+1)l+1}.
\end{eqnarray}
Observe that in \eqref{Eq7a} the singularities of the left hand side are different from those of the right hand side. So,
\begin{eqnarray}\label{Eq7aa}
((d+1)l+2d)W_{k,l+1}-((d+1)l-2d-4)W_{k,l}=(2-l)\widetilde{W}_{k,l}-(l+2)\widetilde{W}_{k,l+1}=0,\;(0\leq l\leq k),
\end{eqnarray}
and by induction on $l$, all the $W_{k,l}$ and $\widetilde{W}_{k,l}$ are null. Thus, \eqref{Dk101} reads

\begin{equation}\label{Dk6a}
a_{n+d+1}-a_{n-2(d+1)}-3(a_n-a_{n-d-1})=0.
\end{equation}
For $n=(d+1)m+r$, $m\geq 0$ and $0\leq r\leq d$, the solutions of \eqref{Dk6a} have the form
\begin{equation}\label{Dk6aa}
a_{(d+1)m+r}=C_{0,r}+C_{1,r}m+C_{2,r}m^2,
\end{equation}
where $C_{0,r}, C_{1,r}$ and $C_{2,r}$ are constants. So, by Corollary~\ref{cor5} we get $T_2=0$.\\ 

{\bf\textit { Case 2:}} There exists $k_{0}\geq 3$ such that $D_{k}=0$ for $k\geq k_{0}$.\\

Suppose that $D_k= T_k^2-T_{k-1}T_{k+1}=0$ for all $k\geq k_0$.
First, notice that if there exists a $k_1\geq k_0$ such that $T_{k_1}=0$, then $T_{k_1-1}T_{k_1+1}=0$. Then, $T_{k_1-1}=0$ or $T_{k_1+1}=0$ and by Corollary~\ref{cor3}, $T_2=0$. We have also $T_{k_0-1}\neq 0$, otherwise $T_{k_0}=0$ and  by Corollary~\ref{cor3}, $T_2=0$.

Now, for $T_k\neq 0$ $(k\geq k_0-1)$, we have 
\begin{equation}
\frac{T_{k+1}}{T_{k}}=\frac{T_{k}}{T_{k-1}}=\frac{T_{k_{0}}}{T_{k_{0}-1}}.
\label{Eq0.4}
\end{equation}%
This means that 
\begin{equation}\label{Eq1.4}
T_{k}=\left( \frac{T_{k_{0}}}{T_{k_{0}-1}}\right)^{k-k_{0}}T_{k_{0}}=ab^{k}
\end{equation}%
where $a=T_{k_{0}-1}^{k_{0}}/T_{k_{0}}^{k_{0}-1}\neq 0$ and $b=T_{k_{0}}/T_{k_{0}-1}\neq 0$.\\
The substitution  $T_{k}=ab^{k}$ in \eqref{gfd1313} for $k\geq k_{0}$\ leads to the equation%
\begin{eqnarray}
&& b\left(a_n-\frac{n-k(d+1)-d}{n-k(d+1)-d+1}a_{n-k(d+1)-d}\right)
+\frac{n+2}{n-d+1} \,c_n^d-\frac{n-k(d+1)+1}{n-k(d+1)+2}\,c_{n-k(d+1)+1}^{d}\nonumber\\
&&=\frac{b^{-k}}{a}\sum_{l=0}^{k-1}\frac{T_{l+1}T_{k-l}}{n-l(d+1)-d+1}=Q_{k}\left( n\right) .  \label{Eq4}
\end{eqnarray}
Let denote (\ref{Eq4}) by $\widetilde{E}\left( k,n\right)$ and make the 
subtraction $\widetilde{E}\left( k+1,n+d+1\right) -\widetilde{E}\left( k,n\right)$ to get
\begin{equation}
b\left( a_{n+d+1}-a_{n}\right)
+\frac{n+d+3}{n+2}c_{n+d+1}^d-\frac{n+2}{n-d+1} \,c_n^d=Q_{k+1}\left(n+d+1\right) -Q_{k}\left( n\right) .
\label{Eq11}
\end{equation}%
On the right hand side of  (\ref{Eq11})  we have, for $k\geq k_{0}$, the expression

\begin{eqnarray}
\widetilde{Q}_{k}\left( n\right) &:=&Q_{k+1}\left(n+d+1\right) -Q_{k}\left( n\right)
=\frac{b^{-k-1}}{a}\sum_{l=0}^{k}\frac{T_{l+1}T_{k+1-l}}{n+d+1-l(d+1)-d+1}-\frac{b^{-k}}{a}\sum_{l=0}^{k-1}\frac{T_{l+1}T_{k-l}}{n-l(d+1)-d+1}\nonumber\\
&=&\frac{b^{-k-1}}{a}\sum_{l=0}^{k}\frac{T_{l+1}T_{k+1-l}}{n-l(d+1)+2}-\frac{b^{-k}}{a}\sum_{l=1}^{k}\frac{T_{l}T_{k+1-l}}{n-l(d+1)+2}\nonumber\\
&=&\frac{b^{-k-1}}{a}\frac{T_{k+1}T_{1}}{n+2}+\frac{b^{-k-1}}{a}\frac{%
T_{k}\left( T_{2}-bT_{1}\right) }{n-d+1}+\frac{b^{-k-1}}{a}\sum_{l=2}^{k}\frac{T_{k+1-l}\left(T_{l+1} -bT_{l}\right)}{n-l(d+1)+2}\nonumber\\
&=&\frac{T_{1}}{n+2}+\frac{T_{2}-bT_{1}}{b(n-d+1)}+\frac{b^{-k-1}}{a}\sum_{l=2}^{k}\frac{T_{k+1-l}\left(T_{l+1} -bT_{l}\right)}{n-l(d+1)+2}.
\end{eqnarray}
from which we deduce
\begin{eqnarray}
\widetilde{Q}_{k+1}\left( n\right)
&=&\frac{T_{1}}{n+2}+\frac{T_{2}-bT_{1}}{b(n-d+1)}+\frac{b^{-k-2}}{a}\sum_{l=2}^{k+1}\frac{T_{k+2-l}\left(T_{l+1} -bT_{l}\right)}{n-l(d+1)+2}
.
\end{eqnarray}
Now  since the left hand side of equation  (\ref{Eq11}) is independent of $k$, it follows
\begin{equation}
\widetilde{Q}_{k+1}\left( n\right) -\widetilde{Q}_{k}\left(n\right) =\frac{b^{-k-2}}{a}\sum_{l=2}^{k}\frac{\left( T_{k-l+2}-bT_{k-l+1}\right)\left( T_{l+1}-bT_{l}\right) }{n-l(d+1)+2}=0.
\end{equation}
As a result, for $2\leq l\leq k$ and $k\geq k_{0}$, we have
\begin{equation}
\left( T_{k-l+2}-bT_{k-l+1}\right)\left( T_{l+1}-bT_{l}\right) 
=0.
\end{equation}
Let take $k=2\left( k_{0}-2\right) -1$ and $l=k_{0}-2$\ to get
$\left(
T_{k_{0}-1}-bT_{k_{0}-2}\right) ^{2}=0$ and then
$T_{k_{0}-1}=bT_{k_{0}-2},$ (or equivalently $D_{k_0-1}=0$).
Thus, the equations \eqref{Eq0.4} and \eqref{Eq1.4} are valid for $k=k_{0}-1$ and by induction we arrive at $T_{4}=bT_{3},$ (or equivalently $D_{4}=0$). For  $k=4$,  the right-hand side of \eqref{gfd23} is null. Consequently, $V_{4,2}=0$ and using $T_5=T_4^2/T_3$ (from $D_{4}=0$) we get $D_3=0$. On the other side (when $T_2\neq 0$) we can write
$$T_{k}=\left(\frac{T_3}{T_2}\right)^{k-2}T_{2}=ab^k,\;\;\text{for } k\geq 2,$$
where $b=T_3/T_2\neq 0$ and $a=T_2^3/T_3^2\neq 0$. Therefore, the equation \eqref{Eq4} reads
\begin{eqnarray}
&& b\left(a_n-\frac{n-k(d+1)-d}{n-k(d+1)-d+1}a_{n-k(d+1)-d}\right)
+\frac{n+2}{n-d+1} \,c_n^d-\frac{n-k(d+1)+1}{n-k(d+1)+2}\,c_{n-k(d+1)+1}^{d}\nonumber\\
&&=\frac{T_1}{n-(d+1)k+2}+\frac{T_1}{n}+\sum_{l=1}^{k-2}\frac{ab}{n-l(d+1)-d+1},\;\;  k\geq 2 \hbox{ and } n\geq (d+1)k+d.\label{gf1222}
\end{eqnarray}
When $n=(d+1)k+d$ and $n=(d+1)(k+1)$, the equation \eqref{gf1222} gives
\begin{equation}\label{gf12222}
ba_{(d+1)k+d}+\frac{(d+1)k+d+2}{(d+1)k+1}c_{(d+1)k+d}^d=\frac{d+1}{d+2}c_{d+1}^d+\frac{T_1}{d+2}+\frac{T_1}{(d+1)k+d}+\sum_{l=1}^{k-2}\frac{ab}{(d+1)(k-l)+1}
\end{equation}
and
\begin{eqnarray}
ba_{(d+1)(k+1)}+\frac{(d+1)(k+1)+2}{(d+1)(k+1)-d+1}c_{(d+1)(k+1)}^d=\frac{a_1b}{d+1}&+&\frac{d+2}{d+3}c_{d+2}^d+\frac{T_1}{d+3}+\frac{T_1}{(d+1)(k+1)}\nonumber\\&+&\sum_{l=1}^{k-2}\frac{ab}{(d+1)(k+1-l)-d+1}\label{gf12223}
\end{eqnarray}
respectively. Let take $n=(d+1)N+d$ in \eqref{gf1222} and use \eqref{gf12222} to obtain the expression
\begin{eqnarray}
&&\frac{d+1}{d+2}c_{d+1}^d+\frac{T_1}{d+2}+\frac{T_1}{(d+1)N+d}+\sum_{l=1}^{N-2}\frac{ab}{(d+1)(N-l)+1}\nonumber\\
&&-\frac{(d+1)(N-k)b}{(d+1)(N-k)+1}a_{(d+1)(N-k)}-\frac{(d+1)(N-k+1)}{(d+1)(N-k)+d+2}c_{(d+1)(N-k+1)}^d\nonumber\\
&&=\frac{T_1}{(d+1)(N-k)+d+2}+\frac{T_1}{(d+1)N+d}+\sum_{l=1}^{k-2}\frac{ab}{(d+1)(N-l)+1}.\nonumber
\end{eqnarray}
In this last equality let put   $N-k$  instead of $k$ to get
\begin{eqnarray}
&&-\frac{b(d+1)k}{(d+1)k+1}a_{(d+1)k}-\frac{(d+1)(k+1)}{(d+1)(k+1)+1}c_{(d+1)(k+1)}^d=\nonumber\\&=&-\frac{(d+1)}{d+2}c_{d+1}^d-\frac{T_1}{d+2}-\sum_{l=1}^{N-2}\frac{ab}{(d+1)(N-l)+1}+\frac{T_1}{(d+1)k+d+2}+\sum_{l=1}^{N-k-2}\frac{ab}{(d+1)(N-l)+1}\nonumber\\
&=&-\frac{d+1}{d+2}c_ {d+1}^d-\frac{T_1}{d+2}+\frac{T_1}{(d+1)k+d+2}-\sum_{l=2}^{k+1}\frac{ab}{(d+1)l+1}.\label{gf12224}
\end{eqnarray}
After defining
$A_1=\frac{a_1}{d+1}+\frac{d+2}{(d+3)b}c_{d+2}^d+\frac{T_1}{(d+3)b}$
, $A_2=-\frac{d+1}{d+2}\frac{c_ {d+1}^d}{b}-\frac{T_1}{(d+2)b}$ and $A_3=\frac{T_1}{b}$, the operation $$\frac{1}{(d+1)(k+1)+2}\left(\frac{(d+1)(k+1)}{(d+1)(k+1)+1}\,Eq\eqref{gf12223}+\frac{(d+1)(k+1)+2}{(d+1)(k+1)-d+1}\,Eq\eqref{gf12224}\right)$$
leads to
\begin{eqnarray}\label{an_tel}
&&\frac{(d+1)(k+1)}{((d+1)(k+1)+2)((d+1)(k+1)+1)}\,a_{(d+1)(k+1)}-\frac{(d+1)k}{((d+1)k+2)((d+1)k+1)}a_{(d+1)k}=\nonumber\\
&=&{\frac {-A_{{1}}d+dA_{{3}}-A_{{3}}}{ \left( (d+1)k+d+2 \right) d}}+{
\frac {2\,A_{{1}}-A_{{3}}}{(d+1)k+d+3}}+{\frac {A_{{2}}d+A_{{3}}}{
 \left( (d+1)k+2 \right) d}}\nonumber\\
 && +\left(\frac{2}{  (d+1)k+d+3 }- \frac{1}{ (d+1)k+d+2}\right)\sum _{l=1}^{k-2}{\frac {a}{ \left( d+1 \right)  \left( k+1-l \right) 
-d+1}}\nonumber\\
 &&-\frac{1}{  \left( 
d+1 \right)  \left( k+1 \right) -d+1 }\sum _{l=2}^{k+1}{\frac {a}{ \left( d+1 \right) l+1}} 
\nonumber\\
&=&{\frac {-A_{{1}}d+dA_{{3}}-A_{{3}}}{ \left( (d+1)k+d+2 \right) d}}+{
\frac {2\,A_{{1}}-A_{{3}}}{(d+1)k+d+3}}+{\frac {A_{{2}}d+A_{{3}}}{
 \left( (d+1)k+2 \right) d}}\nonumber\\
 && +\left(\frac{2}{  (d+1)k+d+3 }- \frac{1}{ (d+1)k+d+2}\right)\sum _{l=2}^{k-1}{\frac {a}{ \left( d+1 \right)  \left( l+1 \right) 
-d+1}}\nonumber\\
 &&-\frac{1}{  \left( 
d+1 \right)  \left( k+1 \right) -d+1 }\sum _{l=2}^{k+1}{\frac {a}{ \left( d+1 \right) l+1}} 
\nonumber\\
&=&\frac{ B_1}{k+\frac{d+3}{d+1}}+\frac{B_2}{k+\frac{d+2}{d+1}}+\frac { B_3}{k+\frac{2}{d+1}}+\frac{a}{(d+1)^2}\left(\frac{2}{  k+\frac{d+3}{d+1} }- \frac{1}{ k+\frac{d+2}{d+1} }\right) \Psi \left( k+\frac{2}{d+1} \right) \nonumber\\
&&-\frac{a}{ 
(d+1)^2 }\frac{1}{k+\frac{2}{d+1}  }\Psi\left( k+2+\frac{1}{d+1} \right),
\end{eqnarray}
where the Digamma function $\Psi(x)$ as well as the short notations $B_1=\frac{2{A_1}-{ A_3}}{d+1}-\frac{2a}{(d+1)^2}\Psi\left(2+\frac{2}{d+1}\right)$, $B_2=\frac{-dA_1+(d-1)A_3}{d(d+1)}+\frac{a}{(d+1)^2}\Psi\left(2+\frac{2}{d+1}\right)$ and $B_3=\frac{ dA_2+ A_3}{d(d+1)}+\frac{a}{(d+1)^2}\Psi\left(2+\frac{1}{d+1}\right)$  are introduced.
Taking $$U_k=\frac{(d+1)k}{((d+1)k+2)((d+1)k+1)}a_{(d+1)k}$$ and 
\begin{eqnarray}
 G(k+1)
&=&\frac{ B_1}{k+\frac{d+3}{d+1}}+\frac{B_2}{k+\frac{d+2}{d+1}}+\frac { B_3}{k+\frac{2}{d+1}}+\frac{a}{(d+1)^2}\left(\frac{2}{  k+\frac{d+3}{d+1} }- \frac{1}{ k+\frac{d+2}{d+1} }\right) \Psi \left( k+\frac{2}{d+1} \right) \nonumber\\
&&-\frac{a}{ 
(d+1)^2 }\frac{1}{k+\frac{2}{d+1}  }\Psi\left( k+2+\frac{1}{d+1} \right).\end{eqnarray}
then \eqref{an_tel} can be  written in compact form as
$$U_{k+1}-U_k=G(k+1).$$
The later recurrence is easily solved to give $$U_k=U_3+\sum^{k}_{j=4}G(j).$$
By using the formula $\Psi(j+1)=\Psi(j)+1/j$ 
and the relations \cite[Theorems 3.1 and 3.2]{milgram}
\begin{equation}
\sum_{l=0}^{k}\frac{\Psi(l+\alpha)}{l+\beta}+\sum_{l=0}^{k}\frac{\Psi(l+\beta+1)}{l+\alpha}=\Psi(k+\alpha+1)\Psi(k+\beta+1)-\Psi(\alpha)\Psi(\beta),
\end{equation}
\begin{equation}
\sum_{j=0}^{k}\frac{\Psi(j+\beta)}{j+\beta}=\frac{1}{2}\left[\Psi\, '(k+\beta+1)-\Psi\,'(\beta)+\Psi(k+\beta+1) ^2-\Psi(\beta)^2\right],
\end{equation}
we obtain
\begin{eqnarray}\label{a2kasympt}
U_k&=&\frac{(d+1)k}{((d+1)k+2)((d+1)k+1)}a_{(d+1)k}=B_{{1}}\Psi \left( k+{\frac {d+3}{d+1}} \right) +   B_{{2}}\Psi \left( k+{\frac {d+2}{d+1}}
 \right) +   B_{{3}}\Psi
 \left( k+\frac{2}{ d+1} \right) \nonumber\\
 && +{\frac {2a}{ \left( d+1 \right) 
 \left((d+1) k+2\right) }}+\frac{a}{\left( d+1 \right) ^{
2}} \left( \Psi'
 \left( k+{\frac {d+3}{d+1}} \right)  + \left( \Psi \left( k+{\frac {d+3}{d+1}}
 \right)  \right) ^{2}\right) \nonumber\\
 && -\frac{a }{\left( d+1 \right) ^{
2}} \Psi
 \left( k+\frac{2} { d+1 } \right) \Psi \left( k+{\frac {d
+2}{d+1}} \right)+\delta_2,
\end{eqnarray}
where
\begin{eqnarray*}
 \delta_2&=&\frac{a }{\left( d+1 \right) ^{
2}} \Psi \left(\frac{ 2}{d+1} \right) 
\Psi \left( {\frac {d+2}{d+1}} \right)  -   B_{{1}}\Psi \left( {
\frac {d+3}{d+1}} \right)  -   B_{{2}}\Psi \left( {\frac {d+2}{d+1}} \right)
 -B_{{3}}\Psi \left( \frac{2}{ d+1 } \right)-\\
 &&{\frac {a}{d+1}}       -\frac{a}{\left( d+1 \right) ^{
2}}\Psi' \left( {\frac {d+3}{
d+1}}\right)   - \frac{a }{\left( d+1 \right) ^{
2}} \left( \Psi \left( {\frac {d+3}{d+1}}
 \right)  \right) ^{2}-G(1)-G(2)-G(3)+U_3.
\end{eqnarray*}
From \eqref{a2kasympt} we deduce the asymptotic behaviour of $a_{(d+1)k}$ as
$k\to\infty$:
\begin{eqnarray}\label{a2kk}
\,a_{(d+1)k}&=&\left(\delta _{1}\left(k +\frac{3}{d+1}+
\frac{2}{(d+1)^2k}\right)+{\frac {a \left( d+2 \right) }{ \left( d+1 \right) ^{2}}}+\frac12{\frac 
{a \left( d+4 \right) }{ \left( d+1 \right) ^{3}k}}+\frac16 {\frac {ad
 \left( d+2 \right) }{ \left( d+1 \right) ^{4}{k}^{2}}}-\frac14 {\frac {a
d \left( d+2 \right) }{ \left( d+1 \right) ^{5}{k}^{3}}}
+\cdots\right)\ln(k)\nonumber\\
&+&\delta
_{2}(d+1)k+\delta _{3}+\frac{\delta _{4} }{k}+\frac{\delta _{5}
}{k^{2}}+\frac{\delta_6}{k^3}+\cdots
\end{eqnarray}
where  coefficients $\delta _{i}$ are defined by (higher terms are omitted)
\begin{eqnarray}\nonumber
&&\delta _{1}=(d+1)(B_1+B_2+B_3),\\\nonumber
&&\delta_3={\frac {  {d}^{2}+d +4 }{d+1}}B_{{1}}+{\frac {  {
d}^{2} +3}{d+1}} B_{{2}}-{\frac {  d-3  }{d+1
}}B_{{3}}+3\,\delta_{{2}}+\,{\frac {3}{d+1}}\,a
,
\\\nonumber
&&\delta_4={\frac {  5\,{d}^{2}+4\,d+47 }{ 6\left( d+1
 \right) ^{2}}} B_{{1}}+{\frac { 8\,{d}^{2}+d+41  }{
 6\left( d+1 \right) ^{2}}}B_{{2}}-{\frac { {d}^{2}+8\,d-41
  }{6 \left( d+1 \right) ^{2}}}B_{{3}}+{\frac {2}{
d+1}}\delta_{{2}}+{\frac { 9\,d+13  }{2 \left( d+1 \right) ^{3}
}}\,a
,
\\\nonumber
&&\delta_5={\frac {  {d}^{2}+2\,d+9  }{6 \left( d+1
 \right) ^{3}}}B_{{1}}-\,{\frac {  {d}^{3}+2\,{d}^{2}-9 }{ 6\left( d+1 \right) ^{3}}} B_{{
2}}+{\frac {  {d}^{2}+2\,d+9
  }{ 6\left( d+1 \right) ^{3}}}B_{{3}}+\,{\frac {  {d}^{3}+5
\,{d}^{2}-2\,d +6 }{12 \left( d+1 \right) ^{4}}}\,a,
\\\nonumber
&&\delta_6={\frac {  {d}^{2}+2\,d-19  }{60
 \left( d+1 \right) ^{2}}}B_{{1}}+{\frac {  {d}^{3}+18
\,{d}^{2}+13\,d-19  }{ 60\left( d+1 \right) ^{3}}}B_{{2}}+{\frac { {d}^{2}+2\,d-19 }{ 60\left( d+1
 \right) ^{2}}}B_{{3}}-{\frac {7\,{d}^{3}+20\,{d}^{2}-11\,d-8
 }{ 24\left( d+1 \right) ^{5}}}\,a
,\\\nonumber
&&\quad\vdots \nonumber
\end{eqnarray}
At this level we should remark that $\lim_{k\to \infty}a_{2k}=\infty$ for all
$\delta_i$, $i=1,2,3,...,$ since $a\neq 0$. Recall that
$c_n^d=\frac{T_1}{d+1}\left((n-d+1)\frac{a_n}{a_{n-d}}-(n-d)\right),\;\;\text{for } n\geq d
$, then the
equation \eqref{gf12222} can be written as

\begin{equation}\label{phi}
a_{(d+1)k+d}\left(b+\frac{T_1}{d+1}\frac{(d+1)k+d+2}{a_{(d+1)k}}\right)=\phi(k),
\end{equation}
where $$\phi \left( k\right)
=\frac{T_1}{d+1}\frac{((d+1)k)((d+1)k+d+2)}{(d+1)k+1}-bA_2+\frac{bA_3}{(d+1)k+d}+\sum_{l=2}^{k-1}\frac{ab}{(d+1)l+1},
$$
 and the equation \eqref{gf12223} can be written
 \begin{equation}\label{phitild}
 a_{(d+1)(k+1)}\left(b+\frac{T_1}{d+1}\frac{(d+1)(k+1)+2}{a_{(d+1)k+1}}\right)=\tilde{\phi}(k),
 \end{equation}
 where $$\tilde{\phi} \left( k\right)
=\frac{T_1}{d+1}\frac{((d+1)(k+1)+2)((d+1)k+1)}{(d+1)k+2}+bA_1+\frac{bA_3}{(d+1)(k+1)}+\sum_{l=3}^{k}\frac{ab}{(d+1)l-d+1}.
$$
From \eqref{phi} and \eqref{phitild} we have 
\begin{equation}
 \frac{1}{a_{(d+1)(k+1)+d}}=
\frac{1}{\phi \left( k+1\right) }\left(
b+\frac{T_1}{d+1}\frac{(d+1)(k+1)+d+2}{a_{(d+1)(k+1)}}\right)
\label{e3d}
\end{equation}
and
\begin{equation}
\frac{1}{a_{(d+1)k+1}}=\frac{d+1}{T_1}\frac{1}{(d+1)(k+1)+2}\left(\frac{\tilde{\phi} (k)}{a_{(d+1)(k+1)}}-b\right),
\label{e4d}
\end{equation}
which give an explicit formula for $a_{(d+1)(k+1)+d}$ and $a_{(d+1)k+1}$.

$\bullet$ If we suppose  $\lim_{k\to\infty} \frac{a_{(d+1)k}}{(d+1)k}=\infty ,$ then from \eqref{e3d} and \eqref{e4d} we deduce on one side%
\begin{equation}
\lim_{k\to\infty} \frac{(d+1)(k+2)}{a_{(d+1)(k+1)+d}}=\lim_{k\to\infty}
\frac{(d+1)(k+2)}{\phi \left( k+1\right) }\left(
b+\frac{T_1}{d+1}\frac{(d+1)(k+1)+d+2}{a_{(d+1)(k+1)}}\right)=\frac{(d+1)b}{T_{1}}
\label{e1}
\end{equation}
and
\begin{equation}
\lim_{k\to\infty} \frac{(d+1)(k+1)+2}{a_{(d+1)k+1}}=\lim_{k\to\infty}
\frac{d+1}{T_1}\left(\frac{\tilde{\phi} (k)}{a_{(d+1)(k+1)}}-b\right)=-\frac{(d+1)b}{T_{1}}.
\label{e11}
\end{equation}
On the other side, for $n=(d+1)k+2d+1$, \eqref{gf1102} reads
\begin{equation}\label{e2}
\frac{(d+1)T_2}{T_1^2}\left(1\!-\!\frac{(d+1)k}{(d+1)k+1}\frac{a_{(d+1)k}}{a_{(d+1)(k+1)+d}}\right)\!=\!\frac{(d+1)(k+2)}{a_{(d+1)(k+1)+d}}\!-\!\frac{2((d+1)(k+1)+1)}{a_{(d+1)(k+1)}}\!+\!\frac{(d+1)k+2}{a_{(d+1)k+1}}.
\end{equation}
Under the assumption $T_2\neq 0$, \eqref{e2} admits the limit $\infty =0$, as $k\to\infty$,  which exhibit a contradiction.

$\bullet$ Now if $\lim_{k\to\infty} \frac{a_{(d+1)k}}{(d+1)k}=\eta_{1}\neq 0,$ then from \eqref{e3d} and \eqref{e4d} we have 
\begin{equation}
\lim_{k\to\infty} \frac{(d+1)(k+2)}{a_{(d+1)(k+1)+d}}\!=\!\lim_{k\to\infty}
\frac{(d+1)(k+2)}{\phi \left( k+1\right) }\left(
b\!+\!\frac{T_1}{d+1}\frac{(d+1)(k+1)+d+2}{a_{(d+1)(k+1)}}\right)\!=\!
\frac{(d+1)b}{T_{1}}\!+\!\frac{1}{\eta_1}
:=\eta _{2},
\label{e3}
\end{equation}
and
\begin{equation}
\lim_{k\to\infty} \frac{(d+1)(k+1)+2}{a_{(d+1)k+1}}=\lim_{k\to\infty}
\frac{d+1}{T_1}\left(\frac{\tilde{\phi} (k)}{a_{(d+1)(k+1)}}-b\right)=\frac{1}{\eta_1}-\frac{(d+1)b}{T_{1}}=\frac{2}{\eta_1}-\eta_2.
\label{e31}
\end{equation}
By taking the limit in (\ref{e2})  we obtain,
\[
\frac{(d+1)T_2}{T_1^2}\left( 1-{\eta _{1}}{\eta_{2}}\right) =\eta _{2}-\frac{2}{\eta _{1}}+\left(\frac{2}{\eta_1}-\eta_2\right)=0.
\]
If $\eta _{2}=0,$ we have the contradiction $T_2=0$. But if $\eta _{2}\neq 0 ,$ then $\eta _{2}=1/\eta _{1}$. From (\ref{e3}) we get  $
(d+1)b/T_{1}+1/\eta_1=\eta _{2}$, which gives the contradiction $b=0$.\\

$\bullet$ Finally if $\lim_{k\to\infty} \frac{a_{(d+1)k}}{(d+1)k}=0,$ then according to \eqref{a2kk}, we have
$\delta_1=\delta_2=0$ and
\begin{eqnarray}\label{a2kkk}
a_{(d+1)k}&=&\frac{((d+1)k+2)((d+1)k+1)}{(d+1)k}\left(B_{{1}}\Psi \left( k+{\frac {d+3}{d+1}} \right) +   B_{{2}}\Psi \left( k+{\frac {d+2}{d+1}}
 \right) -(B_1+   B_{{2}})\Psi
 \left( k+\frac{2}{ d+1} \right)\right. \nonumber\\
 && +{\frac {2a}{ \left( d+1 \right) 
 \left((d+1) k+2 \right) }}+\frac{a}{\left( d+1 \right) ^{
2}} \left( \Psi'
 \left( k+{\frac {d+3}{d+1}} \right)  + \left( \Psi \left( k+{\frac {d+3}{d+1}}
 \right)  \right) ^{2}\right) \nonumber\\
 && \left.-\frac{a }{\left( d+1 \right) ^{
2}} \Psi
 \left( k+\frac{2} { d+1 } \right) \Psi \left( k+{\frac {d
+2}{d+1}} \right)\right),
\nonumber\\
&=&\!\left({\frac {a \left( d+2 \right) }{ \left( d+1 \right) ^{2}}}+\frac12{\frac 
{a \left( d+4 \right) }{ \left( d+1 \right) ^{3}k}}+\frac16 {\frac {ad
 \left( d+2 \right) }{ \left( d+1 \right) ^{4}{k}^{2}}}-\frac14 {\frac {a
d \left( d+2 \right) }{ \left( d+1 \right) ^{5}{k}^{3}}}
+\cdots\right)\ln(k)\!+\!\delta _{3}\!+\!\frac{\delta _{4} }{k}\!+\!\frac{\delta _{5}
}{k^{2}}\!+\!\frac{\delta_6}{k^3}+\cdots.
\end{eqnarray}
Let write the equation \eqref{e2} as
\begin{equation}\label{e20}
\frac{(d+1)T_2}{T_1^2}\left(1\!-\!\frac{(d+1)k}{(d+1)k+1}\frac{a_{(d+1)k}}{a_{(d+1)(k+1)+d}}\right)\!-\!\frac{(d+1)(k+2)}{a_{(d+1)(k+1)+d}}\!+\!\frac{2((d+1)(k+1)+1)}{a_{(d+1)(k+1)}}\!-\!\frac{(d+1)k+2}{a_{(d+1)k+1}}\!=\!0.
\end{equation}
After multiplying the both sides of the equation \eqref{e20} by $\phi(k+1)a_{(d+1)(k+1)}$ and using \eqref{a2kkk}, \eqref{e3d} and \eqref{e4d}, we get, as $k\to \infty$, 
\begin{equation}
\lambda_1 \ln(k)^2+\lambda_2 \ln(k)+\lambda_3 +\cdots=0
\end{equation}
where
\begin{eqnarray*}
\lambda_1 &=&\left( -\frac { \left( d+2 \right) b}{ \left( d+1 \right) ^3
T_1^2}T_2+\frac {b^2}{ \left( d+1 \right) ^2T_1} \right) a^2
\end{eqnarray*}
and\\
\begin{eqnarray*}
\lambda_2 &=&\left(\frac{ 2b ( d+2 ) a^2}{( d+1 )^3T_1^2}\Psi \left( \frac {2d+3}{d+1} \right)  -3\frac { ( d+3 ) ba^2}{ ( d+1 ) ^2T_1^2}+2\frac {
  bA_2 (d+2)+T_1(d+3)  }{ ( d+1 ) ^2T_1^2}a \right) T_2\\
 &&-\frac{2b^2a^2}{( d+1 ) ^{2}T_1}\Psi \left(\frac {2d+3}{d+1} \right)  +3\frac {b^2a^2}{(d+1)T_1}-2\,\frac { ( bA_2+T_1 ) ba}{(d+1)T_1}.
\end{eqnarray*}
So, we must have $\lambda_i=0$, $i=1,2,3,...$. As $a\neq 0$ and $b\neq 0$, from $\lambda_1=0$ we get 
\begin{equation}\label{T2d}
T_2={\frac { d+1  }{d+2}}bT_1,
\end{equation}
and by replacing  $a=T_2/b^2$ and \eqref{T2d} in the equation $\lambda_2=0$ we obtain
\begin{equation}\label{T2d1}
\frac { d-1  }{ \left( d+2 \right) ^{3}}{T_1}^{2}=0.
\end{equation}
If $d\neq 1$, \eqref{T2d1} gives $T_1=0$ which is a contradiction. The case $d=1$  is already treated, \cite[Theorem 1]{meskzahaf}.\\

{\bf\textit { Case 3:}} For every $k_{0}\geq 3$, there exists infinitely many $k,\kappa\geq k_{0}$ such that: $D_{k}=0$ and $D_{\kappa}\neq 0$.\\
We take $k_1$ and $k_2$, $k_1\neq k_2$, with $D_{k_1}=0, D_{k_1+1}\neq 0$, $D_{k_2}=0$ and $D_{k_2+1}\neq 0$ to get from \eqref{gfd23} that
\begin{equation}
a_n-\frac{n-k_1(d+1)-2d-1}{n-k_1(d+1)-2d}a_{n-k_1(d+1)-2d-1}\;\; \text{and}\;\; a_n-\frac{n-k_2(d+1)-2d-1}{n-k_2(d+1)-2d}a_{n-k_2(d+1)-2d-1}
\end{equation}
are two rational functions of $n$.
Consequently, an analogous reasoning  to that of Lemma~\ref{cor6} completes the proof.

In the next subsection we give some expressions concerning the sequence $\{\gamma_{n}^d\}_{n\geq d}$ and the power series $F(t)$. The later is expressed by hypergeometric series \eqref{F01}.

\subsection{ Expressions for $\gamma_{n}^d$ and $F(t)$}

\begin{proposition}
The $d$-symmetric $d$-PS, $\{P_n\}$, generated by \eqref{gf00} satisfies
\begin{eqnarray}\label{DD}
\left\{
\begin{array}{l}
xP_n(x)=P_{n+1}(x)+\gamma_n^d P_{n-d}(x),\quad n\geq
0,\\
P_{-n}(x)=0,\;\;1\leq n\leq d,\;\;\text{and }P_0(x)=1
\end{array}
\right.
\end{eqnarray} 
with
\begin{equation}\label{gm1}
\gamma_{n}^d=\frac{T_1}{d+1}\left((n-d+1)\frac{\alpha_{n-d+1}}{\alpha_{n+1}}-(n-d)\frac{\alpha_{n-d}}{\alpha_n}\right),\;\; \text{for }n\geq d,
\end{equation}
and for $n=dm+r$, $m\geq 1$, $1\leq r\leq d$ we have 
\begin{equation}\label{gm2}
\gamma_{dm+r}^d=\frac{T_1}{d+1}\frac{(dm+r)!(\beta_r(m-r-1)+(d+1)b_r)}{(d(m-1)+r)!\prod_{l=1}^{r}(\beta_lm+b_l)\prod_{l=r}^{d}(\beta_l(m-1)+b_l)}
\end{equation}
with $\gamma_{d}^d=T_1d!/\prod_{l=1}^{d}b_l$,  $b_l=(l+1)\alpha_{l+1}/\alpha_l$ and $\beta_l=b_{d+l}-b_l$, $1\leq l\leq d$.
\end{proposition}

\begin{proof}
The equation \eqref{cdn} is
\begin{equation}\label{gf11aa}
c_{n}^{d} =\frac{\alpha_{n}}{\alpha_{n-d}}\gamma_{n}^d=\frac{T_{1}}{d+1}\left((n-d+1)\frac{a_{n}}{a_{n-d}}-(n-d)\right),\;\; \text{ for } n\geq d.
\end{equation}
As $a_n=\alpha_{n}/\alpha_{n+1}$ we get \eqref{gm1}.\\
According to Theorem~\ref{Th5} we have $R(t)=T_{1}t^{d+1}/(d+1)$. So, by Corollary~\ref{cor06} we obtain
\begin{equation}\label{gf11a}
c_{n}^{d} =\frac{\alpha_{n}}{\alpha_{n-d}}\gamma_{n}^d=\frac{T_{1}}{d+1}\left((n+1)\frac{b_{n-d}}{b_n}-(n-d)\right),\;\;\text{ for } n\geq d+1,
\end{equation} 
with $b_n=(n+1)/a_n=(n+1)\alpha_{n+1}/\alpha_n$ and
\begin{equation}\label{gf11ab}
b_{md+r}=\beta_rm+b_{r},\;\text{ for }\; m\geq 0,\; 1\leq r\leq d,
\end{equation}
where $\beta_r=b_{d+r}-b_r$, for $1\leq r\leq d$.\\
The equation \eqref{gf11a} gives
\begin{equation}\label{gm3}
\gamma_{n}^d=\frac{T_{1}}{d+1}\frac{\alpha_{n-d}\left((n+1)b_{n-d}-(n-d)b_n\right)}{\alpha_{n}b_n},\;\;\text{ for } n\geq d+1.
\end{equation}
We calculate $\alpha_nb_n/\alpha_{n-d}$ by using the relation $\alpha_n=b_{n-1}\alpha_{n-1}/n=\prod_{l=0}^{n-1}b_{l}/n!$ to find 
\begin{equation}\label{gm4}
\frac{\alpha_nb_n}{\alpha_{n-d}}=\frac{(n-d)!}{n!}\frac{\prod_{l=0}^{n}b_{l}}{\prod_{l=0}^{n-d-1}b_{l}}=\frac{(n-d)!}{n!}\prod_{l=n-d}^{n}b_{l}.
\end{equation}
Now for $n=md+r$, we can write
\begin{equation}\label{gm5}
\prod_{l=n-d}^{n}b_{l}=\prod_{l=r}^{d}b_{(m-1)d+l}\prod_{l=1}^{r}b_{md+l}=\prod_{l=r}^{d}(\beta_l(m-1)+b_l)\prod_{l=1}^{r}(\beta_lm+b_l)
\end{equation}
and
\begin{eqnarray}\label{gm6}
(n+1)b_{n-d}-(n-d)b_n&=&(md+r+1)b_{(m-1)d+r}-(md+r-d)b_{md+r}\nonumber\\
&=&(md+r+1)(\beta_r(m-1)+b_r)-(md+r-d)(\beta_rm+b_r)\nonumber\\
&=&\beta_r(m-r-1)+(d+1)b_r.
\end{eqnarray}
Finally, \eqref{gm1} gives $$\gamma_{d}^d=\frac{T_1\alpha_1}{(d+1)\alpha_{d+1}}=\frac{T_1}{d+1}\prod_{l=1}^{d}\frac{\alpha_l}{\alpha_{l+1}}=\frac{T_1d!}{\prod_{l=1}^{d}b_l}$$
and \eqref{gm2} follows by combining \eqref{gm4}, \eqref{gm5} and \eqref{gm6}. 
\end{proof}

\begin{proposition}\label{propF}
If $\prod_{l=1}^{d}\beta_l\neq 0$ then $F(t)=1+F_1(t)+F_2(t)$, where
\begin{eqnarray}\label{F1}
F_1(t)=\alpha_d t^d{}_{d+1}F_{d}\left(
\begin{array}{lll}
1,&\frac{b_d}{\beta_d},&\left(\frac{b_{l+d}}{\beta_l}\right)_{l=1}^{d-1}\\
\left(\frac{l+d}{d}\right)_{l=1}^{d}&
\end{array}
;\left(\prod_{l=1}^{d}\beta_l\right)\left(\frac{t}{d}\right)^d\right)
\end{eqnarray}
and
\begin{eqnarray}\label{F2}
F_2(t)=\sum_{r=1}^{d-1}\alpha_rt^r{}_dF_{d-1}\left(
\begin{array}{lll}
\left(\frac{b_{l+d}}{\beta_l}\right)_{l=1}^{r-1},&\left(\frac{b_l}{\beta_l}\right)_{l=r}^{d}\\
\left(\frac{l}{d}\right)_{l=r+1}^{d-1},&\left(\frac{l}{d}\right)_{l=d+1}^{d+r}&
\end{array}
;\left(\prod_{l=1}^{d}\beta_l\right)\left(\frac{t}{d}\right)^d\right).
\end{eqnarray}
Furthermore, if $\tilde{\beta}_d:=b_d-\beta_d=2b_d-b_{2d}\neq 0$, then $F_1(t)$ in \eqref{F1} can be written as
\begin{equation}\label{F3}
F_1(t)=\frac{\alpha_1}{\tilde{\beta}_d}\left[{}_dF_{d-1}\left(
\begin{array}{ll}
\frac{\tilde{\beta}_d}{\beta_d},&\left(\frac{b_l}{\beta_l}\right)_{l=1}^{d-1}\\
\left(\frac{l}{d}\right)_{l=1}^{d-1}&
\end{array}
;\left(\prod_{l=1}^{d}\beta_l\right)\left(\frac{t}{d}\right)^d\right)-1\right].
\end{equation}
\end{proposition}

\begin{proof}

Recall that $\alpha_n=\prod_{l=0}^{n-1}b_{l}/n!$. If $\prod_{l=1}^{d}\beta_l\neq 0$, then for $n=md+r$ and using the expression of $b_{md+r}$, we obtain 
\begin{eqnarray}
\prod_{l=0}^{md+r-1}b_{l}&=&b_0b_1b_2\cdots b_{r-1}b_rb_{r+1}\cdots b_{d+r-1}\cdots b_{(m-1)d+r-1}b_{(m-1)d+r}\cdots b_{md}b_{md+1}\cdots b_{md+r-1}\nonumber\\
&=&(b_0b_1b_{2}\cdots b_{r-1}) (b_rb_{r+d}\cdots b_{r+(m-1)d})(b_{r+1}\cdots b_{r+1+(m-1)d})\cdots(b_{d+r-1}b_{2d+r-1}\cdots b_{md+r-1})\nonumber\\
&=&\prod_{l=0}^{r-1}b_l\prod_{l=r}^{d+r-1}b_lb_{l+d}\cdots b_{l+(m-1)d}\nonumber\\
&=&r!\alpha_r\prod_{l=r}^{d}b_lb_{l+d}\cdots b_{l+(m-1)d}\prod_{l=d+1}^{d+r-1}b_lb_{l+d}\cdots b_{l+(m-1)d}\nonumber\\
&=&r!\alpha_r\left(\prod_{l=r}^{d}\beta_l\right)^m\prod_{l=r}^{d}\left(\frac{b_l}{\beta_l}\right)_m\left(\prod_{l=1}^{r-1}\beta_l\right)^m\prod_{l=1}^{r-1}\left(1+\frac{b_l}{\beta_l}\right)_m\nonumber\\
&=&r!\alpha_r\left(\prod_{l=1}^{d}\beta_l\right)^m\prod_{l=r}^{d}\left(\frac{b_l}{\beta_l}\right)_m\prod_{l=1}^{r-1}\left(\frac{b_{l+d}}{\beta_l}\right)_m.
\end{eqnarray}
We have also, for $m\geq 0$ and $0\leq r\leq d-1$, the expressions \cite[Lemma 3.3]{BenRomdane}
\begin{eqnarray*}
(md+r)!=r!m!d^{dm}\prod_{l=r+1}^{d-1}\left(\frac{l}{d}\right)_m\prod_{l=d+1}^{d+r}\left(\frac{l}{d}\right)_m
\end{eqnarray*}
and
\begin{eqnarray*}
((m+1)d)!=d!d^{dm}\prod_{l=1}^{d}\left(1+\frac{l}{d}\right)_m.
\end{eqnarray*}
So,
\begin{equation}\label{F4}
\alpha_{md+r}=\frac{\alpha_r\prod_{l=1}^{r-1}\left(\frac{b_{l+d}}{\beta_l}\right)_m\prod_{l=r}^{d}\left(\frac{b_l}{\beta_l}\right)_m}{m!\prod_{l=r+1}^{d-1}\left(\frac{l}{d}\right)_m\prod_{l=d+1}^{d+r}\left(\frac{l}{d}\right)_m}\left(\left(\prod_{l=1}^{d}\beta_l\right)\left(\frac{1}{d}\right)^{d}\right)^m,\;\;\text{for}\; 1\leq r\leq d-1,
\end{equation}
\begin{equation}\label{F5}
\alpha_{md+d}=\frac{\alpha_d\prod_{l=1}^{d-1}\left(\frac{b_{l+d}}{\beta_l}\right)_m\left(\frac{b_d}{\beta_d}\right)_m}{\prod_{l=1}^{d}\left(1+\frac{l}{d}\right)_m}\left(\left(\prod_{l=1}^{d}\beta_l\right)\left(\frac{1}{d}\right)^{d}\right)^m
\end{equation}
and, if $\tilde{\beta}_d\neq 0$,
\begin{equation}\label{F6}
\alpha_{md+d}=\frac{\alpha_1}{\tilde{\beta}_d}\frac{\prod_{l=1}^{d-1}\left(\frac{b_{l}}{\beta_l}\right)_{m+1}\left(\frac{\tilde{\beta}_d}{\beta_d}\right)_{m+1}}{(m+1)!\prod_{l=1}^{d-1}\left(\frac{l}{d}\right)_{m+1}}\left(\left(\prod_{l=1}^{d}\beta_l\right)\left(\frac{1}{d}\right)^{d}\right)^{m+1}.
\end{equation}
Now, expanding $F(t)$ as
\begin{eqnarray*}
F(t)&=&\sum_{n\geq 0}\alpha_nt^n=1+\sum_{m\geq 0}\alpha_{md+d}t^{md+d}+\sum_{r=1}^{d-1}\sum_{m\geq 0}\alpha_{md+r}t^{md+r},
\end{eqnarray*}
the expressions \eqref{F1}, \eqref{F2} and \eqref{F3} follow from \eqref{F4}, \eqref{F5} and \eqref{F6}, respectively.
\end{proof}

\begin{remark}
In proposition~\ref{propF} two expressions of $F(t)$ are given. The first, when $\prod_{l=1}^{d}\beta_l\neq 0$, Equations \eqref{F1} and \eqref{F2}, from which we can deduce the other limiting cases by tending to zero at least a constant $\beta_r$, $1\leq r\leq d$. So, we can enumerate $2^d$ expressions of $F(t)$ similar to that given in \cite[Theorem 3.1]{BenRomdane}. The second, when $\prod_{l=1}^{d}\beta_l\neq 0$ and $\tilde{\beta}_d=b_d-\beta_d\neq 0$, Equations \eqref{F3} and \eqref{F2}, seems to be a new representation of $F(t)$. Similarly, the other limiting cases can be obtained by tending to zero at least a constant $\beta_r$, $1\leq r\leq d$, and $\tilde{\beta}_d$. So, in this second representation, we can enumerate $3\cdot2^{d-1}(=2^{d}+2^{d-1})$ expressions of $F(t)$. We note that, the resulting $2^{d-1}$ expressions when $\tilde{\beta}_d\to 0$, i.e. $\beta_d\to b_d$, are special cases of the first representation when $\beta_d=b_d$. See the illustrative examples given below.  
\end{remark}

\begin{example}

If $d=1$ then $R(t)=T_1t^2/2$, $b_{n}=(b_2-b_1)n+2b_1-b_2=\beta_1 n+\tilde{\beta}_1$, for $n\geq 1$, and $F(t)$, $\gamma_n^d$ have the expressions:\\
1) If $\beta_1 \neq 0$ we have
\begin{equation}\label{F7}
F(t)\equiv F^{\beta_1}(t)=1+\alpha_1 t\text{ }{}_{2}F_{1}\left(
\begin{array}{lll}
1,&\frac{b_1}{\beta_1}\\
 2&
\end{array}
;\beta_1t\right) 
\end{equation}
with $\gamma_1^1=T_1/b_1$ and for $n\geq 2$,
$$\gamma_n^1=\frac{T_1}{2}\frac{n(\beta_1(n-1)+2\tilde{\beta}_1)}{(\beta_1n+\tilde{\beta}_1)(\beta_1(n-1)+\tilde{\beta}_1)}.$$
The limiting case is
\begin{equation}\label{F8}
\lim_{\beta_1\to 0}F^{\beta_1}(t)=1+\alpha_1 t\text{ }{}_{1}F_{1}\left(
\begin{array}{lll}
1\\
 2&
\end{array}
;b_1t\right)=1+\frac{\alpha_1}{b_1}\left(e^{b_1 t}-1\right)
\end{equation}
with $\gamma_n^1=T_1n/b_1$ for $n\geq 1$.\\
2) If $\tilde{\beta}_1\beta_1 \neq 0$, \cite{anshelev,meskzahaf},
\begin{equation}\label{F9}
F(t)\equiv F^{\beta_1,\tilde{\beta}_1}(t)=1+\frac{\alpha_1}{\tilde{\beta}_1}\left((1-\beta_1 t)^{-\frac{\tilde{\beta}_1}{\beta_1}}-1\right)\;\;\;\text{(Ultraspherical polynomials)}
\end{equation}

with
$$\gamma_n^1=\frac{T_1}{2}\frac{n(\beta_1(n-1)+2\tilde{\beta}_1)}{(\beta_1n+\tilde{\beta}_1)(\beta_1(n-1)+\tilde{\beta}_1)}, \text{ for } n\geq 1.$$
The limiting cases are
\begin{equation}\label{F10}
\lim_{\beta_1\to 0}F^{\beta_1,\tilde{\beta}_1}(t)=1+\frac{\alpha_1}{b_1}\left(e^{b_1 t}-1\right) \;\;\;\text{(Hermite polynomials)}
\end{equation}
and
\begin{equation}\label{F11}
\lim_{\tilde{\beta}_1\to 0}F^{\beta_1,\tilde{\beta}_1}(t)=1-\frac{\alpha_1}{\beta_1}\ln(1-\beta_1 t) \;\;\;\text{(Chebyshev polynomials of the first kind)}.
\end{equation}
Remark that \eqref{F9} and \eqref{F11} are special cases of \eqref{F7} for $\tilde{\beta}_1\beta_1 \neq 0$ and $\tilde{\beta}_1=0$,i.e. $\beta_1=b_1$, respectively. Also, \eqref{F8} is exactly \eqref{F10}, since in this case $\beta_1\to b_1$.
\end{example}

\begin{example}
For $d\geq 1$ we take $b_n=\alpha n+\beta$ for all $n\geq 1$. So, from \eqref{gf11ab} we have
$ \beta_r=d\alpha$ and, of course, $b_r=\alpha r+\beta$, for $1\leq r\leq d$. Thus, for $\alpha\beta\neq 0$, $\gamma_n^d$ becomes
\begin{equation}\label{gmm2}
\gamma_{n}^d=\frac{T_1\alpha^{-d-1}}{d+1}\frac{n!(\alpha(n-d)+(d+1)\beta)}{(n-d)!(n+\frac{\beta}{\alpha}-d)_{d+1}},\;\; n\geq d,
\end{equation}
and for $F(t)$, since $\alpha_n=\prod_{l=0}^{n-1}b_l/n!$, we obtain (see \cite{anshelev} for calculations)
\begin{equation}\label{FF9}
F(t)\equiv F^{\alpha,\beta}(t)=1+\frac{\alpha_1}{\beta}\left((1-\alpha t)^{-\beta/\alpha}-1\right).
\end{equation}
Let $\lambda=\beta/\alpha$. Then for $T_1=\alpha^{d+1}(d+1)^{-d}$ and with the change of variable $t\to (d+1)t/\alpha$ in the generating function $\left(1-\alpha(xt-T_1t^{d+1}/(d+1))\right)^{-\lambda}$, we meet the Humbert polynomials \cite{Humbert} generated by $\left(1-(d+1)xt+t^{d+1}\right)^{-\lambda}$. For $d=1$ we have the ultraspherical polynomials. \\

The limiting cases are\\

1. $\alpha \to 0$ and $\beta\neq 0$:

\begin{equation}\label{FF10}
\lim_{\alpha\to 0}F^{\alpha,\beta}(t)=1+\frac{\alpha_1}{\beta}\left(e^{\beta t}-1\right)
\end{equation}
with $\gamma_{n}^d=T_1\beta^{-d}n!/(n-d)!$, for $n\geq d$. In the generating function $\exp\left(\beta\left(xt-T_1t^{d+1}/(d+1)\right)\right)$, with $T_1=\beta^d\left((d+1)d!\right)^{-1}$ and the change of variable $t\to t/\beta$,  we find the generating function $\exp\left(xt-(d+1)^{-2}t^{d+1}/d!\right)$ of the Gould-Hopper polynomials \cite{Gould}  .\\

2. $\beta \to 0$ and $\alpha\neq 0$

\begin{equation}\label{FFa11}
\lim_{\beta\to 0}F^{\alpha,\beta}(t)=1-\frac{\alpha_1}{\alpha}\ln(1-\alpha t) 
\end{equation}
with $\gamma_{d}^d=T_1\alpha^{-d}$ and $\gamma_{n}^d=T_1\alpha^{-d}/(d+1)$ for $n\geq d+1$.\\
Let $b=T_1\alpha^{-d}/(d+1)$. Then by the shift $n\to n+d$ in \eqref{DD}, these polynomials satisfy
\begin{eqnarray}\label{DD1}
\left\{
\begin{array}{l}
P_{n+d+1}(x)=xP_{n+d}(x)-b P_{n}(x),\quad n\geq 1,\\
P_{n}(x)=x^{n},\;\;0\leq n\leq d,\;\;\text{and }P_{d+1}(x)=xP_{d}(x)-(d+1)b P_{0}(x),
\end{array}
\right.
\end{eqnarray} 
 where we recognise the monic Chebyshev $d$-OPS of the first kind generated by (see \cite[Theorem 5.1]{Bencheikh_C})
\begin{equation}\label{Cheby}
\frac{1-dbt^{d+1}}{1-xt+bt^{d+1}}=\sum_{n\geq 0}P_n(x)t^n.
\end{equation}
Remark that
\begin{equation}\label{Cheby1}
\int_{0}^{t}\frac{1}{t}\left(\frac{1-dbt^{d+1}}{1-xt+bt^{d+1}}-1\right)dt=-\ln\left(1-xt+bt^{d+1}\right).
\end{equation}
Then, by changing the variable $t\to \alpha t$, multiplying by $\alpha_1/\alpha$ and adding 1 in  \eqref{Cheby1}, we get the generating function (with $F(t)$ as in \eqref{FFa11}),
\begin{equation}\label{Cheby2}
1-\frac{\alpha_1}{\alpha}\ln\left(1-\alpha\left(xt-\frac{T_1}{d+1}t^{d+1}\right)\right)=1+\sum_{n\geq 1}\frac{\alpha_1}{\alpha  n}P_n(x)t^n.
\end{equation}
  
\end{example}

\begin{example}
For $d=2$ we have $R(t)=T_1t^3/3$ and from \eqref{gm2} we get, for $m\geq 1$, the two expressions
\begin{equation}
\gamma_{2m+2}^{2}=\frac{T_{1}}{3}\,{\frac { 2(m+1)( 2\,m+1 ) ( \beta_{{2}}m+3\tilde{\beta}_2 ) }{ ( \beta_1m+b_{{1}}
 )  ( \beta_2m+b_2 )  ( \beta_2m+\tilde{\beta}_2 ) }}
\end{equation}
and
\begin{equation}
\gamma_{2m+1}^{2}=\frac{T_{1}}{3}\,{\frac {2m ( 2\,m+1 ) ( \beta_{{1}}(m-2)+3
\,b_{{1}}) }{ ( \beta_{{1}}m+b_{{1}} ) 
  ( \beta_{{1}}(m-1)+b_{{1}
} ) ( \beta_{{2}}m+\tilde{\beta}_2 ) }},
\end{equation}
with $\gamma_{2}^{2}=2T_{1}/(b_1b_2)$, $\beta_1=b_{3}-b_1$, $\beta_2=b_{4}-b_2$ and $\tilde{\beta}_2=2b_2-b_4$.\\
We enumerate the following forms of $F(t)$:

\textbf{A}. The first representation by \eqref{F1} and \eqref{F2}.\\

If $\beta_1\beta_2\neq 0$, then

\begin{eqnarray}\label{exp2_A}
F(t)=1+\alpha_2t^{2}\,{}_3F_{2}\left(
\begin{array}{lll}
1,&\frac{b_2}{\beta_2},&\frac{b_3}{\beta_1}\\
\frac{3}{2},&2
\end{array}
;\beta_1\beta_2\left(\frac{t}{2}\right)^2\right)
+\alpha_1t\;{}_2F_{1}\left(
\begin{array}{ll}
\frac{b_1}{\beta_1},&\frac{b_2}{\beta_2}\\
\frac{3}{2}&
\end{array}
;\beta_1\beta_2\left(\frac{t}{2}\right)^2\right).
\end{eqnarray}
The limiting cases are obtained when: $\beta_1\to 0$, $\beta_2\to 0$ or $\left(\beta_1,\beta_2\right)\to (0,0)$.

\textbf{B}. The second representation, by \eqref{F2} and \eqref{F3}, with its limiting cases.\\

1. If $\tilde{\beta}_2\beta_1\beta_2\neq 0$ we have 
\begin{eqnarray}
F(t)=1+\frac{\alpha_1}{\tilde{\beta}_2}\left[{}_2F_{1}\left(
\begin{array}{ll}
\frac{\tilde{\beta}_2}{\beta_2},&\frac{b_1}{\beta_1}\\
\frac{1}{2}&
\end{array}
;\beta_1\beta_2\left(\frac{t}{2}\right)^2\right)-1\right]
+\alpha_1t\;{}_2F_{1}\left(
\begin{array}{ll}
\frac{b_1}{\beta_1},&\frac{b_2}{\beta_2}\\
\frac{3}{2}&
\end{array}
;\beta_1\beta_2\left(\frac{t}{2}\right)^2\right).
\end{eqnarray}

2. If $\beta_1\to 0$:

\begin{eqnarray}
F(t)=1+\frac{\alpha_1}{\tilde{\beta}_2}\left[{}_1F_{1}\left(
\begin{array}{l}
\frac{\tilde{\beta}_2}{\beta_2}\\
\frac{1}{2}
\end{array}
;b_1\beta_2\left(\frac{t}{2}\right)^2\right)-1\right]
+\alpha_1t\;{}_1F_{1}\left(
\begin{array}{l}
\frac{b_2}{\beta_2}\\
\frac{3}{2}
\end{array}
;b_1\beta_2\left(\frac{t}{2}\right)^2\right).
\end{eqnarray}
with 
\begin{eqnarray}
\gamma_{2m}^{2}&=&\frac{T_{{1}}}{3b_{{1}}}\,{\frac { 2m( 2\,m-1 ) ( \beta_{{2}}(m-1)+3\,\tilde{\beta}_2 ) }{  ( \beta_{{2}}m+\tilde{\beta}_2 )  ( \beta_{{2}}(m-1)+\tilde{\beta}_2 ) }},\;\;m\geq 1\\
\gamma_{2m+1}^{2}&=&\frac{T_{1}}{b_{1}}\frac{(2m)(2m+1)}{\beta_2 m+\tilde{\beta}_2},\;\;m\geq 1,
\end{eqnarray}

3. If $\beta_2\to 0$: 

\begin{eqnarray}
F(t)=1+\frac{\alpha_1}{b_2}\left[{}_1F_{1}\left(
\begin{array}{l}
\frac{b_1}{\beta_1}\\
\frac{1}{2}
\end{array}
;b_2\beta_1\left(\frac{t}{2}\right)^2\right)-1\right]
+\alpha_1t\;{}_1F_{1}\left(
\begin{array}{l}
\frac{b_1}{\beta_1}\\
\frac{3}{2}
\end{array}
;b_2\beta_1\left(\frac{t}{2}\right)^2\right).
\end{eqnarray}
with
\begin{eqnarray}
\gamma_{2m}^{2}&=&\frac{T_1}{b_2}\frac{(2m)(2m-1)}{\beta_1 (m-1)+b_{1}},\;\;m\geq 1\\
\gamma_{2m+1}^{2}&=&\frac{T_{1}}{3b_2}\,{\frac {2m ( 2\,m+1 ) ( \beta_{{1}}(m-2)+3
\,b_{{1}}) }{ ( \beta_{{1}}m+b_{{1}} ) 
  ( \beta_{{1}}(m-1)+b_{{1}
} )  }}
,\;\;m\geq 1,
\end{eqnarray}
4. If $\tilde{\beta}_2\to 0$: 

\begin{eqnarray}\label{exp2_4}
F(t)=1+\alpha_2 t^2{}_3F_{2}\left(
\begin{array}{lll}
1&1&1+\frac{b_1}{\beta_1}\\
2&\frac{3}{2}&
\end{array}
;\beta_1b_2\left(\frac{t}{2}\right)^2\right)
+\alpha_1t\;{}_2F_{1}\left(
\begin{array}{ll}
1&\frac{b_1}{\beta_1}\\
\frac{3}{2}&
\end{array}
;\beta_1b_2\left(\frac{t}{2}\right)^2\right).
\end{eqnarray}
with $\gamma_{2}^{2}=2T_{1}/(b_1b_2)$,
\begin{eqnarray}
\gamma_{2m}^{2}&=&\frac{T_1}{3b_2}\frac{2(2m-1)}{(\beta_1 (m-1)+b_{1})},\;\;m\geq 2\\
\gamma_{2m+1}^{2}&=&\frac{T_{1}}{3b_{2}}\,{\frac {2 ( 2\,m+1 ) ( \beta_{{1}}(m-2)+3
\,b_{{1}}) }{ ( \beta_{{1}}m+b_{{1}} ) 
  ( \beta_{{1}}(m-1)+b_{{1}
} )  }}
,\;\;m\geq 1.
\end{eqnarray}
Clearly \eqref{exp2_4} is \eqref{exp2_A} with $\beta_2=b_2$.\\

5. If $\beta_1\to 0$ and $\beta_2\to 0$:

\begin{eqnarray}
F(t)&=&1+\frac{\alpha_1}{b_2}\left[{}_0F_{1}\left(
\begin{array}{l}
- \\
\frac{1}{2}
\end{array}
;b_1b_2\left(\frac{t}{2}\right)^2\right)-1\right]
+\alpha_1t\;{}_0F_{1}\left(
\begin{array}{l}
-\\
\frac{3}{2}
\end{array}
;b_1b_2\left(\frac{t}{2}\right)^2\right)\nonumber\\
&=&1+ \frac{\alpha_1}{b_2}\left(   \cosh\left(\sqrt {b_{{1}}b_2}\,t \right)-1  +\sqrt{\frac { b_2}{b_1}}\,\sinh \left( \sqrt {b_{{1}}b_2}\,t \right)
 \right).
 \end{eqnarray}
with
\begin{eqnarray}
\gamma_{n}^{2}&=&\frac{T_{1}}{b_{1}b_2}n(n-1),\;\;n\geq 2,
\end{eqnarray} 

6. If $\beta_1\to 0$ and $\tilde{\beta}_2\to 0$: 

\begin{eqnarray}
F(t)=1+\alpha_2 t^2{}_2F_{2}\left(
\begin{array}{ll}
1&1\\
2&\frac{3}{2}
\end{array}
;b_1b_2\left(\frac{t}{2}\right)^2\right)
+\alpha_1t\;{}_1F_{1}\left(
\begin{array}{l}
1\\
\frac{3}{2}
\end{array}
;b_1b_2\left(\frac{t}{2}\right)^2\right).
\end{eqnarray}
with $\gamma_{2}^{2}=2T_{1}/(b_1b_2)$,
\begin{eqnarray}
\gamma_{2m}^{2}&=&\frac{2T_{1}}{3b_{1}b_2}(2m-1),\;\;m\geq 2\\
\gamma_{2m+1}^{2}&=&\frac{2T_{1}}{b_{1}b_2}(2m+1),\;\;m\geq 1.
\end{eqnarray}

\end{example}

{\bf Acknowledgements:} We would like to thank Dr. Yanallah Abdelkader for precious help and useful discussions.

\end{document}